\documentclass{article}
\usepackage{graphicx} 
\usepackage{tabularx}
\usepackage{amsfonts,array,mathdots}
\usepackage{amsthm}
  \usepackage{yfonts}
  \usepackage{mathdots}
  \usepackage{float}
  \usepackage[%
    dvipsnames,
    svgnames,
    x11names,
    fixpdftex
  ]{xcolor}
  \usepackage[toc,titletoc]{appendix}
\usepackage{fdsymbol}
  \usepackage{multirow}
  \usepackage{multicol}
  \usepackage{colortbl}
  \usepackage{numprint}
  \usepackage{textcomp}
  \usepackage{booktabs}
  \usepackage{enumitem}
 \usepackage{etoolbox}

\usepackage{mathtools} 
\usepackage[backend=biber]{biblatex}
\newcommand{\Sn}{\operatorname{S}}
\newcommand{\SO}{\operatorname{SO}}
\newcommand{\Bl}{\operatorname{BL}}
\newcommand{\Lo}{\operatorname{Lo}}
\newcommand{\Bn}{\operatorname{B}}
\newcommand{\Fl}{\operatorname{Flag}}
\newcommand{\Qt}{\operatorname{Quat}}
\newcommand{\Spin}{\operatorname{Spin}}
\newcommand{\Na}{\operatorname{N}}
\newcommand{\Blc}{\operatorname{BLC}}
\newcommand{\inv}{\operatorname{inv}}

\newcommand{\block}{\operatorname{block}}

\newcommand{\diag}{\operatorname{Diag}}

\newcommand{\Cl}{\operatorname{Cl}}
\newcommand{\Hquat}{\operatorname{HQuat}}
\newcommand{\Tr}{\operatorname{Trace}}

\newcommand{\sig}{\operatorname{sign}}
\newcommand{\NL}{\operatorname{NL}}
\newcommand{\GL}{\operatorname{GL}}
\newcommand{\Up}{\operatorname{Up}}
\newcommand{\Bru}{\operatorname{Bru}}
\newcommand{\Pos}{\operatorname{Pos}}
\newcommand{\BLS}{\operatorname{BLS}}

\theoremstyle{plain}

\newtheorem{Theo}{Theorem}
\newtheorem{Conje}{Conjecture}

\newtheorem{Fact}{Fact}
\newtheorem{Prop}{Proposition}[section]
\newtheorem{Lem}{Lemma}[section]

\theoremstyle{definition}

\AtEndEnvironment{Ex}{\null\hfill$\diamond$}%
\newtheorem{Def}[Prop]{Definition}
\newtheorem{Remark}[Prop]{Remark}
\AtEndEnvironment{Remark}{\null\hfill$\diamond$}%
\newtheorem{Ex}{Example}[section]

\newcommand{\sbd}{\smallblackdiamond}
\newcommand{\sd}{\diamond}
\newcommand{\sbc}{\bullet}

\title{Homotopy Type of Intersections of Real Bruhat Cells in Dimension 6}
\author{
Em{\'\i}lia Alves\thanks{
Em{\'\i}lia Alves,
emiliaalves@id.uff.br,
Departamento de Matem\'atica Aplicada,
Instituto de Matem\'atica e Estat{\'\i}stica,
Universidade Federal Fluminense,
Rua Professor Marcos Waldemar de Freitas Reis s/n,
Niterói, RJ 24210-201, Brazil}
\and
Giovanna Leal\thanks{
Giovanna Leal, giovannaleal@puc-rio.br,
Departamento de Matem\'atica, PUC-Rio,
Rua Marqu\^es de S\~ao Vicente 225,
Rio de Janeiro, RJ 22451-900, Brazil}
}

\date{June 2026}

\begin{document}

\maketitle

\begin{abstract}
We study the homotopy type of the intersection of two real Bruhat cells. 
This homotopy type coincides with that of an explicit submanifold of the group of real lower triangular matrices with diagonal entries equal to 1.
For $(n+1)\times(n+1)$ matrices with $n\leq4$, these submanifolds are disjoint unions of contractible connected components. Our focus is on such intersections for $6\times6$ real matrices. For this, we study the connected components of Bruhat cells corresponding to permutations $\sigma\in \Sn_6$ with at most 12 inversions, using the structure of the associated dual CW complexes. New combinatorial and topological tools are developed to describe the structure of the spaces $\mathcal{\Bl}_\sigma$ for certain permutations.
As a consequence, we show that, among permutations with at most 12 inversions, all connected components are contractible except for $\sigma=[563412]$. For this permutation, we identify a new non-contractible connected component with the homotopy type of the circle.

\end{abstract}

\section{Introduction}

Let $\GL_{n+1}$ denote the group of invertible real $(n+1)\times (n+1)$ matrices and let $\Up_{n+1}\subset\GL_{n+1}$ be the subgroup of upper triangular matrices. For each permutation $\sigma\in\Sn_{n+1}$, let $P_{\sigma}$ denote the corresponding permutation matrix, defined by $(P_{\sigma})_{i,i^{\sigma}}=1$ and $0$ otherwise. The top permutation is denoted by $\eta$, and is given by $i^{\eta}=n+2-i$, for all $i\in[\![n+1]\!]=\{1,\dots,n+1\}$. 

Let $\Lo_{n+1}^{1}$ be the nilpotent group of real lower triangular matrices with diagonal entries equal to 1. 
The Bruhat decomposition induces a partition $\Lo_{n+1}^{1}$ into subsets $\Bl_{\sigma}$ for $\sigma\in\Sn_{n+1}$:
\begin{equation}\label{Blsigma}
\Bl_{\sigma}=\{L\in \Lo^{1}_{n+1}\, |\, \exists U_0,U_1 \in \Up_{n+1},\, L=U_0P_{\sigma}U_1 \}.   
\end{equation}
The set $\Bl_{\sigma}$ has the same homotopy type as the intersection of two appropriate open Bruhat cells in $\GL_{n+1}$, corresponding to different bases.

In \cite{alves2022onthehomotopy}, Alves and Saldanha introduced useful tools for studying the homotopy type of these intersections. 
They apply this tools to prove the following theorem: 

\begin{Fact}\label{imrn}
Consider $\sigma\in\Sn_{n+1}$ and $\Bl_{\sigma}\subset\Lo_{n+1}^{1}$.
\begin{enumerate}
    \item For $n\leq 4$, every connected component of every set $\Bl_{\sigma}$ is contractible.
    \item For $n=5$ and $\sigma=[563412]\in\Sn_{6}$, there exist connected components of $\Bl_{\sigma}$, which are homotopically equivalent to $\mathbb{S}^{1}$.
    \item For $n\geq 5$, there exist connected components of $\Bl_{\eta}$, which have even Euler characteristic.
\end{enumerate}
\end{Fact}

Our aim is to extend this construction to the case $n=5$. Specifically, we examine the connected components of the set $\Bl_{\sigma}$, for $\sigma\in\Sn_6$. 
The main result of this paper is the following:

\begin{Theo}\label{theo1}
        Consider $\sigma\in\Sn_{6}$ and $\Bl_{\sigma}\subset \Lo^{1}_{6}$.
    \begin{enumerate}
        \item For $\inv(\sigma)\leq11$, every component of every set $\Bl_{\sigma}$ is contractible;
        \item For $\inv(\sigma)=12$, except for $\sigma=[563412]$, every component of every set $\Bl_{\sigma}$ is contractible;
        \item For $\sigma=[563412]$, the set $\Bl_{z}$ has
        \begin{enumerate}[label=(\alph*)]
        \item 8 values of $z\in\acute{\sigma}\Qt_6$ where there are five contractible connected components: 4 thin and 1 thick;
        \item 32 values of $z\in\acute{\sigma}\Qt_6$ where there are a single contractible connected component;
        \item 4 values of $z\in\acute{\sigma}\Qt_6$ where there are two connected components homotopically equivalent to $\mathbb{S}^1$;
        \item 16 values of $z\in\acute{\sigma}\Qt_6$ where there are a single connected component homotopically equivalent to $\mathbb{S}^1$;
        \item 4 values of $z\in\acute{\sigma}\Qt_6$ where there are a single inconclusive connected component, with Euler characteristic equal to 1.        
        \end{enumerate} 
        
        \end{enumerate}
\end{Theo}

According to Theorem 2 in \cite{alves2022onthehomotopy}, for $\sigma\in\Sn_{n+1}$ there exist a finite CW complex $\Blc_{\sigma}$ homotopically equivalent to $\Bl_{\sigma}$.
In particular, the connected components of $\Blc_{\sigma}$ correspond precisely to those of $\Bl_{\sigma}$.

Therefore, to determine the homotopy type of $\Bl_{\sigma}$, for $\sigma\in\Sn_6$, we classify the permutations by their number of inversions and study the connected components of $\Blc_{\sigma}$.
In this case, the maximum number of inversions is 15. This paper covers the case up to 12 inversions.
Analyzing the components using our current method becomes increasingly challenging as the dimension of the ancestries grows. 

\begin{Remark}
 For permutations with more than 12 inversions, we already have partial results. For 13 and 14 inversions, the Euler characteristic is either 0 or 1. For 15 inversions, there exists a component with Euler characteristic 2, while the others are 0 or 1. Further details are provided in Section~\ref{131415}.
\end{Remark}

Over the past century, Bruhat cell decompositions have been important to mathematics, particularly in the study of Grassmannians and flag spaces, and have become standard tools in fields like topology, enumerative geometry, representation theory, and the study of locally convex curves. 
Despite their long-standing importance, the topological study of Bruhat cell intersections, whether in pairs or more complex collections, remains relatively underexplored. 
These intersections naturally arise in various mathematical areas, including singularity theory, Kazhdan-Lusztig theory, and matroid theory. However, detailed topological results on these intersections are still scarce (see \cite{shapiro1997combinatorics}).

One notable exception to this lack of topological insight is the problem of counting connected components in pairwise intersections of big Bruhat cells over the real numbers. 
Significant advances were made in this area during the late 1990s, with key contributions found in works such as \cite{shapiro1998skew}, \cite{gekhtman2003number} and \cite{shapiro1997connected}.

The intersection of two opposite big Bruhat cells in $\Fl_{n+1}$ is homeomorphic to $\Bl_{\eta}$, where $\eta\in\Sn_{n+1}$ is the top permutation. The number of connected components of $\Bl_{\eta}$ is $2,6,20,$ and $52$ for $n=1,2,3,4$, respectively. For $n\geq 5$, the number of connected components stabilizes and is given by $3.2^n$. 

The relative positions of two big Bruhat cells in $\Fl_{n+1}$ correspond bijectively to the elements of $\Sn_{n+1}$.
In particular, opposite big Bruhat cells are associated with the top permutation $\eta\in\Sn_{n+1}$. 
The study of the number of connected components in the intersection of two big cells for a given relative position $\sigma$ was initiated in \cite{shapiro1998skew}. 
For any specific $\sigma\in\Sn_{n+1}$, the number of connected components can be determined based on the results from \cite{seven2005orbits}. 
However, to the best of our knowledge, no closed formula has been found.

Section \ref{reviwe} introduces pivotal concepts relevant to this work, including the wiring diagram, which will be widely used throughout the paper, and provides a concise overview of the key structures and ideas underlying this research. These include matrix groups such as $\Qt_{n+1}$, $\Spin_{n+1}$, and $\tilde{\Bn}_{n+1}^{+}$; the Clifford algebra $\Cl_{n+1}^{0}$; the notions of preancestry and ancestry; Bruhat cells, their properties and stratification; and the associated CW complexes. A solid understanding of these concepts, particularly in the context of the wiring diagram, is essential for what follows.

In Section \ref{wdd}, we develop a collection of new lemmas that play a crucial role in our analysis. These results provide the combinatorial and topological tools necessary to describe the structure of the spaces $\Bl_{\sigma}$ and to advance the study of CW complexes associated with permutations.

Sections \ref{inv6} and \ref{inv12} examine the connected components of $\Bl_{\sigma}$ for $\sigma \in \Sn_6$ with $\inv(\sigma) \leq 12$, excluding the permutation $\sigma = [563412]$. For all such permutations, we show that every connected component of $\Bl_{\sigma}$ is contractible. 

Section 6 focuses on the permutation $\sigma = [563412]$. This permutation has one component homotopically equivalent to $\mathbb{S}^1$ composed of 2-dimensional cells, as presented in \cite{alves2022onthehomotopy}. In addition, there are five contractible connected components and one component with Euler characteristic equal to~1, whose homotopy type remains inconclusive. Finally, we construct a new component, also homotopically equivalent to $\mathbb{S}^1$, but composed of 3-dimensional cells.

To conclude, Section \ref{131415} discusses partial results concerning the remaining 20 permutations in $\Sn_6$.

\vspace{1.0cm}

\textit{Acknowledgements.} This paper is based on the Ph.D. thesis of Giovanna Leal, who was advised by Nicolau C. Saldanha and co-advised by Emília Alves. The authors are grateful to Nicolau C. Saldanha for his guidance, encouragement, and many valuable suggestions throughout the development of this work. We also thank the members of the Ph.D. committee for valuable comments and suggestions, with special thanks to Michael Shapiro. The authors acknowledge the financial support of CAPES and PUC-Rio.

\section{Review}\label{reviwe}

We recall the notation and main results from~\cite{alves2022onthehomotopy} that will be used throughout the paper.

\subsection{Permutation}

This section starts with a review of the various forms used to represent a permutation $\sigma\in\Sn_{n+1}$. A common one is a list of values that is called the \textit{complete notation} $\sigma=~ [1^{\sigma}2^{\sigma}\ldots (n+1)^{\sigma}] \in \Sn_{n+1}$. 
Another way is by using Coxeter-Weyl generators $a_i=(i,i+1)$, with $i\in[\![n+1]\!]=\{1,\dots,n+1\}$, where a inversion is a pair $(i,j)\in[\![n+1]\!]^2$ with $i<j$ and $i^{\sigma}>j^{\sigma}$. 
Using this notation, a permutation can be written as a product of these transpositions. 
We denote by $\inv(\sigma)$ the number of inversions of $\sigma$.
A \textit{reduced word} for a permutation $\sigma \in \Sn_{n+1}$ is an expression of $\sigma$ as a product of generators $a_{i}=(i,i+1)$, where the number of generators is the number of inversions of sigma. For instance, $\sigma = [4321] = a_1a_2a_1a_3a_2a_1=a_2a_1a_2a_3a_2a_1$. 
Usually, there is more than one reduced word for a given permutation $\sigma$, but we shall keep our word fixed in the construction of the stratification of $\Bl_{\sigma}$.

A reduced word can be represented by a \textit{wiring diagram}, Figure \ref{wiring} illustrate an example. Each crossing in the diagram corresponds to a generator $a_{i}$, we read the diagram from left to right and from top to bottom.

\begin{figure}[H]
    \centering
    \includegraphics[scale=0.11]{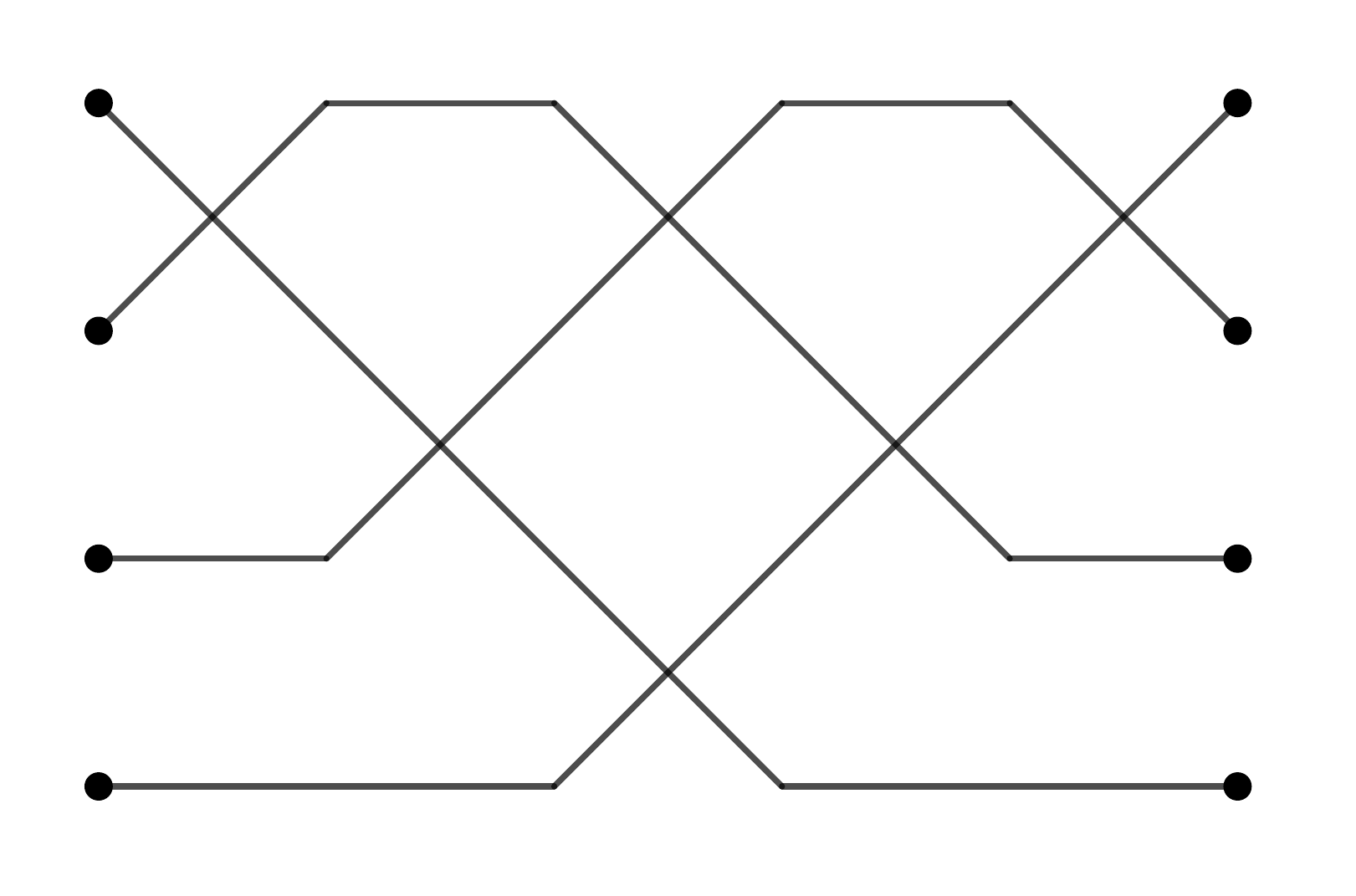}
    \caption{Wiring diagram of $\sigma=a_{1}a_{2}a_{1}a_{3}a_{2}a_{1}\in\Sn_4$.}
    \label{wiring}
\end{figure} 

Note that the inversion $a_i=(i,i+1)$  appears in the wiring diagram at height $i+\frac{1}{2}$.
The horizontal row between the starting points of two adjacent wires at height $i+\frac{1}{2}$ is called $r_i$.

Throughout this paper, we adopt the notations and constructions introduced in \cite{alves2022onthehomotopy}, in particular the notions of ancestry and the sets $\BLS_{\varepsilon}$, which we mention here before their precise definitions.
An \textit{ancestry} is a wiring diagram with black and white squares and disks marking the crossings. The ancestries must satisfy several conditions, which we will discuss in detail in the next section. The number of black or white squares are equal, called \textit{dimension} of the ancestry, $d=\dim(\varepsilon)=\#\diamond=\#\blackdiamond$. See Figure \ref{ancestries} for examples of ancestries.

\begin{figure}[H]
    \centering
    \includegraphics[scale=0.5]{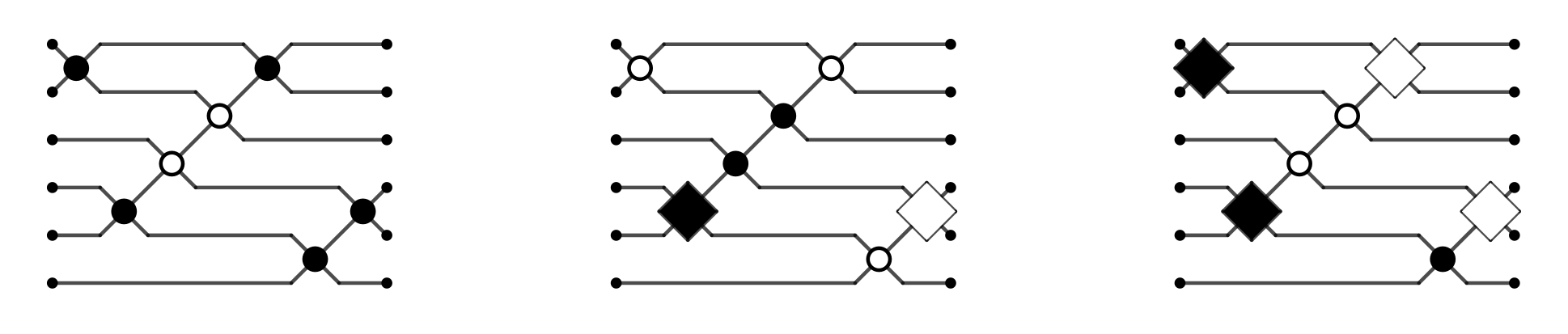}
    \caption{Examples of ancestries for $\sigma=a_{1}a_{4}a_{3}a_{2}a_{1}a_{5}a_{4}\in\Sn_3$ with dimension 0, 1 and 2, respectively.}
    \label{ancestries}
\end{figure} 

Following \cite{alves2022onthehomotopy}, for each ancestry $\varepsilon$ we define $\BLS_{\varepsilon}$ the set of matrices with ancestry $\varepsilon$. Each $\BLS_{\varepsilon}$ is a smooth submanifold of codimension $d=\dim(\varepsilon)$, such that 
\[
\Bl_{\sigma} = \bigsqcup_{\varepsilon} \BLS_{\varepsilon}.
\] 
Furthermore, $\BLS_{\varepsilon}$ is diffeomorphic to $\mathbb{R}^{\inv(\sigma)-d}$.

\subsection{Preancestry and ancestry}

In order to construct an ancestry, we first introduce an auxiliary structure, called a preancestry. Given a permutation $\sigma\in\Sn_{n+1}$ and a fixed reduced word $\sigma=a_{i_1}\cdots a_{i_\ell}$, with $\ell=\inv(\sigma)$. 
A \textit{preancestry} for a reduced word is a sequence $(\rho_k)_{0\leq k\leq \ell}$ of permutations with the following properties:
\begin{enumerate}
    \item $\rho_0=\rho_\ell=\eta$;
    \item for all $k\in[\![n]\!]$, either $\rho_k=\rho_{k-1}$ or $\rho_k=\rho_{k-1}a_{i_k}$;
    \item for all $k\in[\![n]\!]$, if $\rho_{k-1}a_{i_k}>\rho_{k-1}$ then $\rho_k=\rho_{k-1}a_{i_k}$.
\end{enumerate}

We emphasize that throughout this paper, inequalities between permutations refer to the (strong) Bruhat order in $\Sn_{n+1}$ (see \cite{bjorner2005combinatorics} and \cite{goulart2021locally}).

A preancestry $(\rho_k)$ is often more conveniently described by the sequence $\varepsilon_0:[\![n]\!]\to\{-2,0,2\}$ given by:
\[
\varepsilon_0(k)=
\begin{cases}
-2, \quad \rho_k=\rho_{k-1}a_{i_k}<\rho_{k-1},\\
0,\quad \rho_k=\rho_{k-1},\\
2, \quad \rho_k=\rho_{k-1}a_{i_k}>\rho_{k-1}.
\end{cases}
\]

In the wiring diagram $\varepsilon_0(k)=-2$ corresponds to a black square, while $\varepsilon_0(k)=+2$ corresponds to a white square. 
Figure \ref{inv7sigma41-pre} shows two preancestries for the permutation $\sigma=a_1a_2a_1\in\Sn_3$. The sequence 
\[
(\rho_0=\eta,\rho_1=\rho_0,\rho_2=\rho_1,\rho_3=\rho_2)=(\eta,\eta,\eta,\eta)
\]
corresponds to the preancestry $\varepsilon_{0_1}=(0,0,0)$, and 
\[
(\rho_0=\eta,
    \rho_1=\eta a_1,
    \rho_2=\rho_1,
    \rho_3=\rho_2 a_1)=(\eta,a_1a_2,a_1a_2,\eta)
\]
    corresponds to $\varepsilon_{0_2}=(-2,0,2)$.
    \begin{figure}[H]
    \centering
    \includegraphics[scale=1.0]{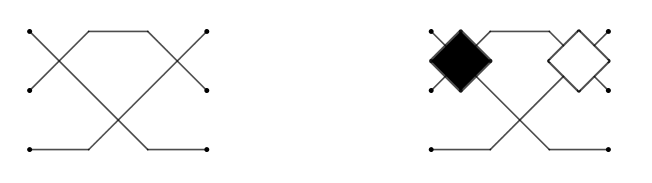}
    \caption{Preancestries $\varepsilon_{0_1}=(0,0,0)$ and $\varepsilon_{0_2}=(-2,0,+2)$.}
    \label{inv7sigma41-pre}
\end{figure}

To define the concept of an ancestry for a permutation, we need to introduce some important matrix groups.

The finite matrix group generated by the elements $\hat{a}_i$ satisfying the conditions: 
\[
\hat{a}_{i}^{2}=-1;\quad \hat{a}_{i}\hat{a}_{j}=\hat{a}_{j}\hat{a}_{i} \text{ if } |i-j|\neq1;\quad \hat{a}_{i}\hat{a}_{j}=-\hat{a}_{j}\hat{a}_{i} \text{ if } |i-j|=1,
\]
is called $\Qt_{n+1}$ and can be written as $\Qt_{n+1}=\Hquat_{n+1}\sqcup(-\Hquat_{n+1})$, where $\Hquat_{n+1}$ consists of elements that appear with a positive sign in $\Qt_{n+1}$. 

The subalgebra of the even elements of $\Cl_{n+1}$ is called Clifford Algebra $\Cl^{0}_{n+1}\subset\Cl_{n+1}$, which is a vector space of dimension $2^{n}$, with the orthonormal basis
\[
\Hquat_{n+1}=\{1,\hat{a}_{1},\hat{a}_{2},\hat{a}_{1}\hat{a}_{2},\hat{a}_{3},\hat{a}_{1}\hat{a}_{3},\hat{a}_{2}\hat{a}_{3},\hat{a}_{1}\hat{a}_{2}\hat{a}_{3},\dots,\hat{a}_{1}\cdots\hat{a}_{n}\}.
\]

It is well known that the spin group, denoted by $\Spin_{n+1}$, is the double universal cover of $\SO_{n+1}$.
Given the generators $\hat{a}_{i}\in\Qt_{n+1}$, we can define the one-parameter subgroups:
\[
\alpha_{i}^{\Spin}: \mathbb{R}\to\Spin_{n+1}
\]
\[
\alpha_{i}^{\Spin}(\theta)\mapsto\exp\left(\theta\frac{\hat{a}_{i}}{2}\right)=\cos\left(\frac{\theta}{2}\right)+\hat{a}_{i}\sin\left(\frac{\theta}{2}\right).\]

The group $\Spin_{n+1}$ is generated by $\alpha_{i}^{\Spin}(\theta)$.
Since $\alpha_{i}^{\Spin}(\pi)=\hat{a}_{i}$, it follows that $\Qt_{n+1}\subset \Spin_{n+1}$. Moreover, given that $\Cl_{n+1}^{0}$ is generated by $\Hquat_{n+1}$, we have $\Qt_{n+1}\subset\Spin_{n+1}\subset\Cl^{0}_{n+1}$.

We define $\tilde{\Bn}^{+}_{n+1}\subset\Spin_{n+1}$ as the subgroup generated by the elements $\{\acute{a}_{1},\dots,\acute{a}_{n}\}$. These generators are given by
\[
\acute{a}_{i}=\alpha_{i}^{\Spin}\left(\frac{\pi}{2}\right)=\frac{1+\hat{a}_{i}}{\sqrt{2}},\quad
\grave{a}_{i}=\alpha_{i}^{\Spin}\left(-\frac{\pi}{2}\right)=\frac{1-\hat{a}_{i}}{\sqrt{2}}.
\]
Given a reduced word for a permutation $\sigma=a_{i_{1}}\cdots a_{i_{\ell}}$, we set $\acute{\sigma}=\acute{a}_{i_{1}}\cdots\acute{a}_{i_{\ell}}$ and $\grave{\sigma}=\grave{a}_{i_{1}}\cdots\grave{a}_{i_{\ell}}$.

There exists a unique group homomorphism $\Pi:\Spin_{n+1}\to\SO_{n+1}$, such that $\alpha^{\Spin}_{i}(\theta)\mapsto \alpha^{\SO}_{i}(\theta)$. In other words, the map is defined by:
\[
\alpha_{i}^{\SO}(\theta) \mapsto
\begin{pmatrix}
I_{1} & & & \\
 & \cos(\theta) & -\sin(\theta) & \\
 & \sin(\theta) & \cos(\theta) &  \\
 & & & I_{2} 
\end{pmatrix},
\]
where $I_{1}\in\mathbb{R}^{(i-1)\times (i-1)}$ and $I_{2}\in\mathbb{R}^{(n-i)\times (n-i)}$ are identity matrices.
We often omit the superscripts when they are clear from the context.

Let $\diag_{n+1}$ denote the group of diagonal matrices with entries in $\{\pm1\}$.
From the map $\Pi:\Spin_{n+1}\to\SO_{n+1}$, it follows that $\Pi[\Qt_{n+1}]=\diag_{n+1}^{+}$ and $\Pi[\tilde{\Bn}_{n+1}^{+}]=\Bn^{+}_{n+1}$, where $\diag_{n+1}\cap\SO_{n+1}=\diag_{n+1}^{+}\subset\Bn_{n+1}$ denotes the normal subgroup of diagonal matrices, and $\Bn^{+}_{n+1}=\Bn_{n+1}\cap\SO_{n+1}$, with $\Bn_{n+1}$ being the group of signed permutation matrices.
The map $\Pi: \Spin_{n+1} \to \SO_{n+1}$ provides the following exact sequences:
\[
1 \to \Qt_{n+1} \to \tilde{\Bn}^{+}_{n+1} \to S_{n+1} \to 1,\quad 1 \to \{\pm1\} \to \Qt_{n+1} \to \diag_{n+1}^{+} \to 1.
\]

Having established these structures, we can now introduce ancestries that provide a combinatorial perspective on the problem. 
The wiring diagram representation of an ancestry was introduced in the previous section; we now present its sequential representation, as given in \cite{alves2022onthehomotopy}.

An \textit{ancestry} for a permutation is a sequence $(\varrho)_{0\leq k \leq \ell}$ of elements of $\tilde{\Bn}^{+}_{n+1}$ such that:
    \begin{enumerate}
        \item $\varrho_0=\acute{\eta},\quad \varrho_\ell\in\acute{\eta}\Qt_{n+1}$;
        \item for all $k$, we have $\varrho_k=\varrho_{k-1}$ or $\varrho_k=\varrho_{k-1}\acute{a}_{i_k}$ or $\varrho_k=\varrho_{k-1}\hat{a}_{i_k}$;
        \item the sequence $(\rho_k)$ defined by $\rho_k=\Pi_{\tilde{\Bn}^{+}_{n+1},\Sn_{n+1}}(\varrho_k)$ is a preancestry.
    \end{enumerate}
The final condition can be restated as follows: if $\Pi(\varrho_{k-1}) < \Pi(\varrho_{k-1})a_{i_k}$, it implies that $\varrho_k = \varrho_{k-1}\acute{a}_{i_k}$, for all $k$.

There are three additional sequences that represent an ancestry. 
Two consists of integers and the other comprises elements of $\Qt_{n+1}$.

The first sequence is $\xi:[\![l]\!]\to\{0,1,2\}$, defined recursively by $\varrho_k=\varrho_{k-1}(\acute{a}_{i_k})^{\xi(k)}$.
It follows that
$\varrho_k=(\acute{a}_{i_1})^{\xi(1)}\cdots(\acute{a}_{i_k})^{\xi(k)}$.
The second sequence $\acute{\rho}_kq_k=\varrho_k$ is determined by the relation $(q_k)_{0\leq k \leq \ell}$, with $q_k\in\Qt_{n+1}$, in particular, $q_\ell=\grave{\eta}\varrho_\ell$.

The third and most used one is $\varepsilon:[\![l]\!]\to \{\pm1,\pm2\}$ given by:
\[
\varepsilon(k)=
\begin{cases}
    -2, \qquad\qquad\qquad\quad\quad \xi(k)=1,\quad \rho_k<\rho_{k-1},\\
    +2, \qquad\qquad\qquad\quad\quad \xi(k)=1,\quad \rho_k>\rho_{k-1},\\
    (1-\xi(k))[\hat{a}_{i_k},q_{k-1}],\hspace{0.5em} \xi(k)\neq1.
\end{cases}
\]

Recall that, in the wiring diagram of $\sigma$, an ancestry $\varepsilon$ is represented using black and white squares and disks. As in the case of preancestries, black and white squares correspond to $-2$ and $+2$, respectively. For $-1$ and $+1$, we use black and white disks, respectively.
The number of black or white squares are equal, called dimension $d=\dim(\varepsilon)=\#\diamond=\#\blackdiamond$. Figure \ref{wiring ancestries} shows examples of an ancestry of dimensions 0 and 1.

\begin{figure}[H]
    \centering
    \includegraphics[scale=0.1]{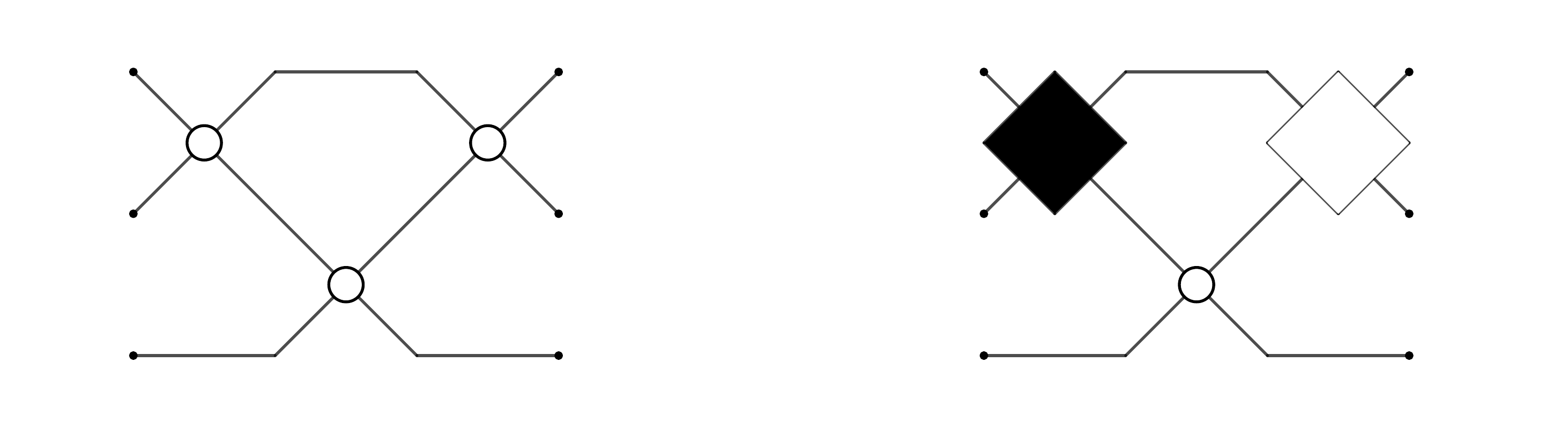}
    \caption{Ancestries $\varepsilon_1=(\circ\circ\circ)$ and $\varepsilon_2=(\sbd\circ\sd)$.}
    \label{wiring ancestries}
\end{figure} 

A region is a bounded connected component of the complement of the wires in a wiring diagram.
It is determined by the vertices $k_1,k_2$ in row $i_{k_1}$, together with all $k$ such that $k_1<k<k_2$ and $|i_k-i_{k_1}|=1$.
If $k_1$ and $k_2$ have opposite signs, flipping the signs along the boundary of this region defines a \textit{click}.

Figure \ref{click} illustrates a click applied to the upper region at the wiring diagram of $\sigma=a_2a_1a_4a_3a_2a_5a_4\in\Sn_6$. 
The left diagram has an ancestry $\varepsilon_1=(\sbc\sbc\sbc\sbc\circ\circ\sbc)$,
while the diagram on the right corresponds to $\varepsilon_2=~(\circ\circ\sbc\circ\sbc\circ\sbc)$.
\begin{figure}[H]
    \centering
    \includegraphics[scale=1.0]{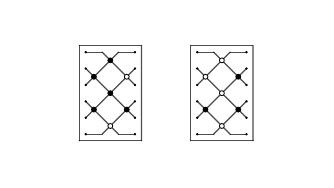}   
    \caption{Example of a click.}
    \label{click}
\end{figure}

An ancestry of dimension 0 is called \textit{thin} if all vertices in the same row have the same sign, and \textit{thick} otherwise. 

\subsection{Counting ancestries}

The group $\diag_{n+1}$, acts on $\SO_{n+1}$ by conjugation. The quotient $\mathcal{E}_n=$\( \displaystyle \frac{\diag_{n+1}}{\pm I} \), where $I$ is the identity matrix, is canonically isomorphic to $\{\pm1\}^{[\![n]\!]}$. A matrix $D\in\diag_{n+1}$ corresponds to $E\in\mathcal{E}_n$ with entries $E_{i}=D_{i,i}D_{i+1,i+1}$.
The group $\mathcal{E}_n$ acts by automorphisms on $\SO_{n+1}$, this action lifts to $\Spin_{n+1}$ and extends to $\Cl^{0}_{n+1}$ with $(\alpha_i(\theta))^{E}=\alpha_i(E_i\theta)$, and $(\hat{a}_i)^{E}=E_i\hat{a}_i.$

For $z\in\Cl^{0}_{n+1}$, define the real part by $\mathfrak{R}(z)=2^{-n}\Tr(z)=\langle z,1\rangle$. Moreover, for all $z\in\tilde{\Bn}_{n+1}^{+}$ and $E\in\mathcal{E}_{n}$, we have $\mathfrak{R}(z^{E})=\mathfrak{R}(z)$.

Let $X$ be a partition of $[\![n+1]\!]$. The subgroup $H_{\diag,X}\leq\diag_{n+1}^{+}$, consists of matrices $E\in~\diag_{n+1}^{+}$ such that for every subset $A=\{i_{1},\dots,i_{k}\}\in X$, the product $E_{i_{1}i_{1}}\cdots E_{i_{k}i_{k}}=1$. Let $H_{X}=\Pi^{-1}[H_{\diag,X}]\leq\Qt_{n+1}$, where $\Pi:\Qt_{n+1}\to\diag_{n+1}^{+}$ is the restriction of $\Pi:\Spin_{n+1}\to\SO_{n+1}$. For a permutation $\sigma\in\Sn_{n+1}$, let $X_{\sigma}$ denote its cycle partition, and define $H_{\sigma}=H_{X_{\sigma}}$. In this case, $|H_{\sigma}|=2^{n+2-c}$, where c is the number of cycles of $\sigma$. Given a preancestry $\varepsilon_0$, the refinement $X_{\varepsilon_0}$ of $X_{\sigma}$ is such that: whenever $\varepsilon_0(k)=0$ and the $k$-th crossing involves $(i_0, i_1)$, the pair $\{i_0, i_1\}$ must be contained in some set $A$ in $X_{\varepsilon_0}$.

For a given $z\in\acute{\sigma}\Qt_{n+1}$, $\NL_{\varepsilon_0}(z)$ is the number of ancestries $\varepsilon$ associated with a given preancestry $\varepsilon_0$, such that $P(\varepsilon)=(\acute{a}_{i_1})^{\sig(\varepsilon(1))}\cdots(\acute{a}_{i_\ell})^{\sig(\varepsilon(\ell))}=z$.
We have 
\begin{equation}\label{sumneg}
    \NL_{\varepsilon_0}(z)-\NL_{\varepsilon_0}(-z)=2^{\frac{\ell-2d}{2}}\mathfrak{R}(z),
\end{equation}
and for $z=qz_0$ with $\mathfrak{R}(z)>0$, we have
\begin{equation}\label{sumpos}
\NL_{\varepsilon_0}(z)+\NL_{\varepsilon_0}(-z)=
\begin{cases}
    2^{\ell-2d+1}/|H_{\varepsilon_0}|, \quad q\in H_{\varepsilon_0},\\
    0, \qquad\qquad\qquad q\notin H_{\varepsilon_0}.
\end{cases}
\end{equation}
The special case where $d=0$ yields:
\begin{equation}\label{sumdim0}
\NL_{\varepsilon_0}(z)=2^{\ell-n+b-1}+2^{\frac{\ell}{2}-1}\mathfrak{R}(z).    
\end{equation}

There are $2^{n-b}$ thin ancestries, where $b=\block(\sigma)$. We assume for now on that $\sigma$ does not block, i.e., $b=0$.

The action of $\mathcal{E}_n$ determines the orbit size. 
If $\mathfrak{R}(z)=0$, the orbit of $z$ has size $2^{n-c+2}$; otherwise, the orbits of $z$ and $-z$ are disjoint, each of size $2^{n-c+1}$.
The isotropy group of $\acute{\sigma}$ has order $2^{\tilde{c}}$ with $c-2\leq \tilde{c} \leq c$, where $c$ is the number of cycles of $\sigma$.
The orbit of $\acute{\sigma}$ under the action of $\mathcal{E}_n$ is $\acute{\sigma}^{\mathcal{E}_n}=\{\acute{\sigma}^{E}, E\in\mathcal{E}_n\}.$
For $z\in~\acute{\sigma}\Qt_{n+1}$, we have 
\begin{equation}\label{sumthin}
\NL_{thin}(z)=
\begin{cases}
    2^{n-\tilde{c}},\hspace{1.0em} z\in\acute{\sigma}^{\mathcal{E}_n},\\
    0, \hspace{2.3em} z\notin\acute{\sigma}^{\mathcal{E}_n},
\end{cases}\quad \text{with}\quad |\acute{\sigma}^{\mathcal{E}_n}|=2^{\tilde{c}}.
\end{equation}

\subsection{Bruhat Stratification}

The Bruhat decomposition of $\SO_{n+1}$, or Bruhat stratification with signs, is
\[
\SO_{n+1}=\bigsqcup _{P\in\Bn^{+}_{n+1}} \Bru_{P},\quad \Bru_{P}=(\Up_{n+1}^{+}P\Up_{n+1}^{+})\cap\SO_{n+1},\quad P\in\Bn^{+}_{n+1}.
\]
Its lift to $\Spin_{n+1}$ gives $\Bru_{\sigma}=\Pi^{-1}[\Bru_{\sigma}^{\SO}]$, which has $2^{n+1}$ connected components, each containing an element $z\in\acute{\sigma}\Qt_{n+1}$.
The connected component containing $z\in\tilde{\Bn}^{+}_{n+1}$ is the smooth contractible submanifold $\Bru_{z}\subset\Spin_{n+1}$ of dimension $\ell=\inv(\sigma)$. Therefore, 
\[
\Bru_{\sigma}=\bigsqcup_{z\in\acute{\sigma}\Qt_{n+1}}\Bru_{z}, \quad \Spin_{n+1}=\bigsqcup_{z\in\tilde{\Bn}_{n+1}^{+}}\Bru_{z}.
\]

The set $\Bru_{\acute{\sigma}}$ can be parametrized using $\alpha_{i_k}(\theta_k)$, where $\sigma=~a_{i_1}\cdots a_{i_k}\in~\Sn_{n+1}$ is a reduced word for $\sigma$, as follows:
\[
\Bru_{\acute{\sigma}}=\{\alpha_{i_1}(\theta_1)\cdots\alpha_{i_k}(\theta_k); \theta_i\in(0,\pi)\}.\]
The strong Bruhat order can be expressed as 
\[
\sigma_0\leq\sigma_1 \iff \Bru_{\sigma_0}\subseteq\overline{\Bru_{\sigma_1}},
\]
and the lift Bruhat order on $\tilde{B}^{+}_{n+1}$ is 
\[
z_0\leq z_1 \iff \Bru_{z_0}\subseteq\overline{\Bru_{z_1}},\quad z_0,z_1\in\tilde{\Bn}^{+}_{n+1}.
\]
Ancestries inherit a partial order from the lifted Bruhat order: given ancestries $\varepsilon$ and $\tilde{\varepsilon}$, with sequences $(\varrho_k)$ and $(\tilde{\varrho_k})$, we write
\[
\varepsilon\preceq\tilde{\varepsilon} \iff (\forall k,\varrho_k\leq\tilde{\varrho}_k).
\]
Thus, $\varepsilon\preceq\tilde{\varepsilon}$ implies $P(\varepsilon)=P(\tilde{\varepsilon})$.
The set $U_{\varepsilon}=\{\tilde{\varepsilon}\,|\, \varepsilon\preceq\tilde{\varepsilon}\}$ is the upper set generated by $\varepsilon$.

For ancestries of dimension 0, $\varepsilon$ is $\preceq$-maximal.
For ancestries with $\dim(\varepsilon)>0$, we define $\tilde{\varepsilon}$ setting $\tilde{\varepsilon}(k)=\sig(\varepsilon(k))$. 
This ensures $\tilde{\varepsilon}\preceq\varepsilon$. 
In a wiring diagram, the ancestry $\tilde{\varepsilon}$ is obtained by replacing each square by a disk of the same color.

When $\dim(\varepsilon)=1$, the upper set $U_{\varepsilon}$ contains $\varepsilon$ and two ancestries of dimension 0. 
One is $\tilde{\varepsilon}=\sig(\varepsilon)$, the other is obtained from $\tilde{\varepsilon}$ performing a click in the region corresponding to $\varepsilon$.

\begin{Ex}\label{d=1}
For $\sigma=[321]=a_1a_2a_1\in\Sn_3$, Figure \ref{upperseteta} shows an ancestry of dimension 1, $\varepsilon=(-2,1,2)$, and the upper set generated by it.

\begin{figure}[H]
    \centering
    \includegraphics[scale=1.2]{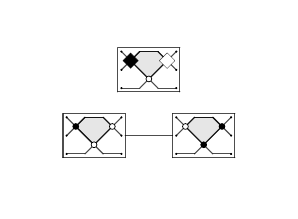}
    \caption{Upper set of $\varepsilon=(\sbd\circ\sd)$.}
    \label{upperseteta}
\end{figure}
The upper set consists of two ancestries of dimension 0 and one of dimension 1: $U_\varepsilon=\{(\sbd\circ\sd),(\sbc\circ\circ),(\circ\sbc\sbc)\}$.
\end{Ex}

When a click can be performed in a region, we generate an ancestry of dimension 1 represented by an edge, connecting two ancestries of dimension 0: one with the same signs as the ancestry of dimension 1, and the other with signs altered by the click.

Following the Bruhat decomposition, the nilpotent group $\Lo_{n+1}^{1}$ can be partitioned as
\[
\Lo_{n+1}^{1}=\bigsqcup_{\sigma\in\Sn_{n+1}}\Bl_{\sigma}, \quad \Bl_{\sigma} = \{L \in \Lo_{n+1}^{1} \mid \exists U_0, U_1 \in \Up_{n+1}, L = U_0 P_{\sigma} U_1\}.
\]

The orthogonal factor $Q \in \SO_{n+1}$ from the QR factorization of $L\in\Lo^{1}_{n+1}$ defines a smooth map $
\mathbf{Q}_{\SO} : \Lo_{n+1}^{1} \to \SO_{n+1}$, with $
\mathbf{Q}_{\SO}(L) = Q$, and its lift $\mathbf{Q}:\Lo_{n+1}^{1}\to\Spin_{n+1}$, with $\mathbf{Q}(I)=1$. The inverse map $\mathbf{L}=\mathbf{Q}^{-1}:\mathcal{U}_1\to\Lo_{n+1}^{1}$, is also a smooth diffeomorphism and corresponds to the $LU$ factorization.

The open contractible neighborhood of the identity element $1\in\Spin_{n+1}$ given by $\mathcal{U}_1=\mathbf{Q}[\Lo_{n+1}^{1}]=\grave{\eta}\Bru_{\acute{\eta}}\subset\Spin_{n+1}$ is a top-dimensional Bruhat cell for the basis described by $\grave{\eta}$. 
Consequently, each $\Bl_{\sigma}$ can be decomposed into $2^{n+1}$ subsets $\Bl_{z}=\mathbf{Q}^{-1}[\Bru_{z}]=\mathbf{Q}^{-1}[\Bru_{z}\cap\mathcal{U}_1]\subseteq\Lo_{n+1}^{1}$, the intersection of two Bruhat cells for different bases.

Let $\lambda_i(t)=\exp(t\mathfrak{l}_i)=I+t\mathfrak{l}_i$ be the one-parameter subgroup of $\Lo_{n+1}^{1}$, with $\mathfrak{l}_i\in\mathfrak{lo}_{n+1}$ the lower triangular matrix with a single nonzero entry $(\mathfrak{l}_i)_{(i+1,i)}=1$.
The group $\mathcal{E}_{n}=\{\pm1\}^{[\![n]\!]}$ acts on $\Lo_{n+1}^{1}$ by $(\lambda_{i}(t))^{E}=\lambda_i(E_it)$.
This action modifies the entries of the matrices in $\Lo_{n+1}^{1}$ according to the signs specified by $E\in\mathcal{E}_n$.
Hence, the sets $\Bl_{z}$ and $\Bl_{z^{E}}$ are diffeomorphic. 

To determine the homotopy type of $\Bl_{\sigma}$, we decompose $\acute{\sigma}\Qt_{n+1}$ into $\mathcal{E}_n$-orbits. For each orbit, we select a representative $z$ and determine the homotopy type of $\Bl_{z}$.

The set $\Pos_{\eta}$ of totally positive matrices is a contractible connected component of $\Bl_{\eta}$, and we have $L\in\Pos_{\sigma}$ if and only if there exist positive numbers $t_1,\ldots,t_\ell$ such that $L=\lambda_{i_1}(t_1)\cdots\lambda_{i_\ell}(t_\ell).$
The set $\Pos_{\sigma}$ is also a contractible connected component of $\Bl_{\sigma}$.
For an ancestry $\varepsilon$ with $\dim(\varepsilon)=0$, define 
\begin{equation}\label{bls0}
    \BLS_{\varepsilon}=\{\lambda_{i_1}(t_1)\cdots\lambda_{i_l}(t_l)\;|\;t_k\in\mathbb{R}\backslash \{0\},\;\sig(t_k)=\varepsilon_k\}\subset\Bl_{\sigma}.
\end{equation}
It follows that $\BLS_{\varepsilon}\subseteq\Bl_{z}$, with $z=P(\varepsilon)=(\acute{a}_{i_1})^{\varepsilon(1)}\cdots(\acute{a}_{i_\ell})^{\varepsilon(\ell)}\in\acute{\sigma}\Qt_{n+1}.$

Contractible components corresponding to thin ancestries are obtained as $\BLS_{\varepsilon}=(\Pos_{\sigma})^{E}$, while the complement $
\Bl_{z,thick}=\Bl_{z}\backslash\bigcup_{\varepsilon \text{ thin}}\BLS_{\varepsilon}$ forms the thick part. 

It is established from \cite{shapiro1998skew} that $\Bl_{\eta}$ comprises $3\cdot2^{n}$ connected components. Additionally, from \cite{alves2022onthehomotopy} we know that for $n\geq5$, the $3\cdot2^n$ connected components of $\Bl_{\eta}$ are $\Pos_{\eta}^{E}$, 
with $E\in\mathcal{E}_n$, and 
$\Bl_{z,\text{ thick}}$, with $z\in\acute{\eta}\Qt_{n+1}.$ 
The first list are the $2^{n}$ thin connected components; the second are the $2^{n+1}$ thick connected components.

Finally, sequences of parameters $(\theta_k)_{0 \leq k \leq \ell}$ and elements $(z_k)_{0 \leq k \leq \ell}$ allow reconstructing ancestries $\varepsilon$ whitin $\Bl_{\sigma}$, so that
\[
\Bl_{\sigma}=\bigsqcup_{\varepsilon}\BLS_{\varepsilon},\qquad \Bl_z=\bigsqcup_{P(\varepsilon)=z}\BLS_{\varepsilon},
\]
where $\varepsilon$ varies over ancestries.
Moreover, $\sig(\varepsilon(k))=\sig(\tilde{\theta}_k)$, with $\varepsilon(k)=-2$ if and only if $\rho_k<\rho_{k-1}; \varepsilon(k)=+2$ if and only if $\rho_k>\rho_{k-1}$.

For any preancestry $\varepsilon_0$ and $z\in\acute{\sigma}\Qt_{n+1}$, we have $\NL_{\varepsilon_0}(z)=N_{\varepsilon_0}(z)$, where $N_{\varepsilon_0}(z)$ is the number of ancestries $\varepsilon$ for which $\BLS_{\varepsilon}\subset\Bl_{z}$, thus Equations \ref{sumneg}, \ref{sumpos} and \ref{sumdim0} hold for $N_{\varepsilon_0}(z)$.

Finally, the set $\BLS_{\varepsilon}$ is a smooth submanifold diffeomorphic to $\mathbb{R}^{\ell-d}$. Furthermore, the union of all $\BLS_{\varepsilon}$ is an open subset of the larger space.

\subsection{CW-Complexes}

The concept behind the CW complex $\Blc_{\sigma}$ is that it behaves as a dual cell structure to the stratification. From \cite{alves2021remarks}, for each $\sigma\in\Sn_{n+1}$, there exists a finite CW complex $\Blc_{\sigma}$ and a homotopy equivalence $i_{\sigma}:\Blc_{\sigma}\to\Bl_{\sigma}$, where the cells $\Blc_{\varepsilon}$ correspond to ancestries $\varepsilon$ and have matching dimension.

The Euler characteristic of $\Bl_z$ can be computed via preancestries:
\[
\chi(\Bl_z)=\sum_{\varepsilon_0}(-1)^{\dim(\varepsilon_0)}N_{\varepsilon_0}(z).
\]

In general, for an upper set $U$ of ancestries, we have 
\[
\BLS_{U}=\bigcup_{\varepsilon\in U}\BLS_{\varepsilon}\subseteq\Bl_{\sigma}, \qquad
\Blc_{U}=\bigcup_{\varepsilon\in U}\Blc_{\varepsilon}\subseteq \Blc_{\sigma}.
\]
Here $\BLS_{U}$ is open, $\Blc_{U}$ is closed and a CW complex, and the restriction $i_{\sigma}|_{\Blc_{U}}:\Blc_{U}\to\BLS_{U}$ is a homotopy equivalence.

The image of the gluing map for $\Blc_{\varepsilon}$ is contained in $\Blc_{U^{*}_{\varepsilon}}$, where $U^{*}_{\varepsilon}=U_{\varepsilon}\backslash\{\varepsilon\}$.
Consider an ancestry $\varepsilon$ with $d=\dim(\varepsilon)>0$. Define two non empty subsets $U^{\pm}_{\varepsilon}\subset U^{*}_{\varepsilon}$. Let $k_{\bullet}$ denote the largest index $k$ such that $\varepsilon(k)=-2$. We have $\varrho_{k_{\bullet}}=\varrho_{k_{\bullet}-1}\acute{a}_{i_k}$. Define $\varrho^{-}_{k_{\bullet}}=\varrho_{k_{\bullet}-1}$ and $\varrho^{+}_{k_{\bullet}}=\varrho_{k_{\bullet}-1}\hat{a}_{i_k}$. For $\tilde{\varepsilon}\in U_{\varepsilon}^{*}$, let $(\tilde{\varrho}_k)_{0\leq k \leq l}$ be defined as usual. Then: 
\[
\tilde{\varepsilon}\in U^{\pm}_{\varepsilon} \iff \tilde{\varrho}_{k_{\bullet}}=\varrho_{k_{\bullet}}^{\pm} \text{ and } \tilde{\varrho}_k=\varrho_k \text{ for } 0\leq k < k_{\bullet}.
\]
These sets $U^{\pm}_{\varepsilon}$ are disjoint.
From that, we then define the sets near $\BLS_{\varepsilon}$ as:
\[
\BLS_{\varepsilon}^{\pm}=\bigcup_{\tilde{\varepsilon}\in U^{\pm}_{\varepsilon}}\BLS_{\tilde{\varepsilon}}.
\]
Lemma 16.4 from \cite{alves2021remarks} describes the sets near $\BLS_{\varepsilon}$.

An ancestry $\varepsilon$ of dimension $d=\dim(\varepsilon)>0$ is called \textit{tame} if $\BLS_{U^{*}_{\varepsilon}}$ is homotopically equivalent to $\mathbb{S}^{d-1}$, and its intersection with $\BLS_{\varepsilon}^{\pm}$ defines generators of $H^{d-1}(\BLS_{U^{*}_{\varepsilon}};\mathbb{Z})\approx\mathbb{Z}$.
If these conditions are not satisfied, $\varepsilon$ is called wild.

In terms of $\Blc_{U^{*}_{\varepsilon}}$, tameness means that $\Blc_{U^{*}_{\varepsilon}}$ is homotopically equivalent to $\mathbb{S}^{d-1}$, and that we can construct cocycles $\omega_{\Blc}^{\pm}\in Z^{d-1}(\Blc_{U^{*}_{\varepsilon}};\mathbb{Z})$ by considering elements of $U^{*}_{\varepsilon}$ of dimension $d-1$ as cells of $\Blc_{U^{*}_{\varepsilon}}$, these cocycles $\omega_{\Blc}^{\pm}$ are generators of $H^{d-1}\approx\mathbb{Z}$.

In this case, the gluing map $\beta:\mathbb{S}^{d-1}\to \Blc_{U^{*}{\varepsilon}}$ is a homotopy equivalence, and attaching the $d$-cell $\Blc_{\varepsilon}$ yields $\Blc_{U_\varepsilon}$.
Up to the present, no examples of wild ancestries have been found.

Therefore, we can state the Fact \ref{imrn} from the Introduction.

\section{Wiring Diagram Decomposition}\label{wdd}

This section begins with a brief review of the block decomposition of wiring diagrams. We then introduce three types of decompositions, called splits, which are a key contribution of this work and offer new tools for studying their properties.

\subsection{Block decomposition}

Recall that $\sigma\in\Sn_{n+1}$ blocks at $j$ if and only if $a_j$ does not appear in a reduced word for $\sigma$.

\begin{Prop}
If $\sigma\in\Sn_{n+1}$ blocks at $j$ 
then there exist permutations $\sigma_0\in\Sn_j$ and $\sigma_1\in\Sn_{n+1-j}$ such that $\sigma=\sigma_0\oplus\sigma_1$.
\end{Prop}

\begin{Ex}
Let $\sigma=[2 3 1 6 4 5]=a_2a_1a_4a_5\in\Sn_6$. 
\begin{figure}[H]
    \centering
    \includegraphics[scale=0.09]{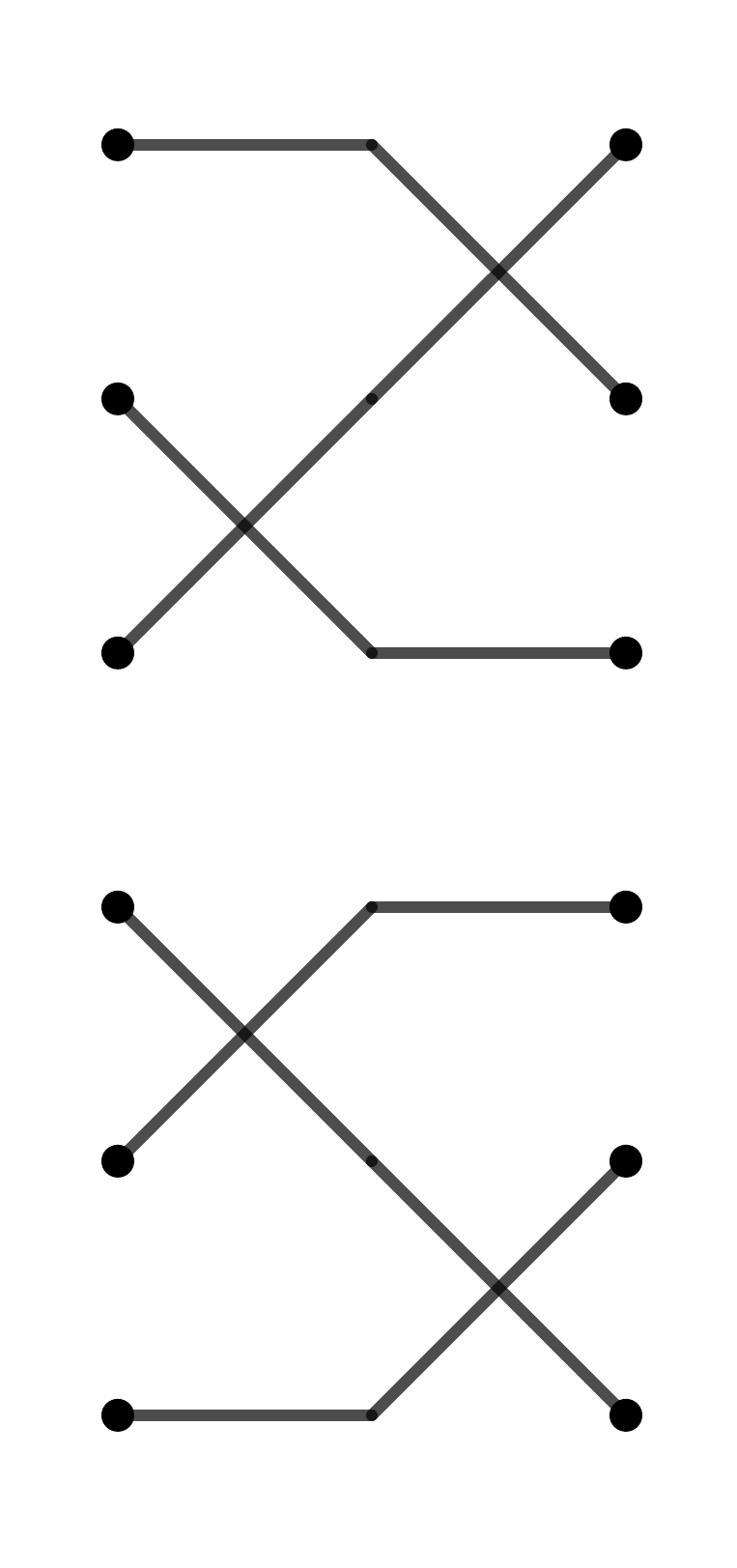}
    \caption{Wiring diagram of $\sigma=a_2a_1a_4a_5\in\Sn_6$.}
    \label{_4.1}
\end{figure} 

Note that $\sigma$ blocks at 3. 
This permutation can be represented as a sum of two permutations: 
$\sigma=\sigma_1\oplus\sigma_2$, 
where $\sigma_1=a_2a_1\in\Sn_3$ and $\sigma_2=a_1a_2\in\Sn_3$.
\end{Ex}

\begin{Lem}\label{blocklemma}
Let $\sigma=a_{i_1}\cdots a_{i_k}\in\Sn_{n+1}$ be a reduced word. 
If $\sigma$ blocks at $j$ such that, 
$\sigma=\sigma_0\oplus\sigma_1$ with $\sigma_0\in\Sn_j$ and $\sigma\in\Sn_{n+1-j}$ 
then $\Bl_{\sigma}=\Bl_{\sigma_0}\oplus\Bl_{\sigma_1}$.
\end{Lem}

\begin{proof}
If $\sigma\in\Sn_{n+1}$ blocks at $j$ such that $\sigma=\sigma_0\oplus\sigma_1$, with $\sigma_0\in\Sn_j$ and $\sigma_1\in\Sn_{n+1-j}$, then the permutation matrix $P_{\sigma}$ has two diagonal blocks, $P_{\sigma_0}$ and $P_{\sigma_1}$, such that 
$P_{\sigma}=P_{\sigma_0}\oplus P_{\sigma_1}$.

Let $L\in\Lo_{n+1}^{1}$. Suppose that there exist $L_0\in\Bl_{\sigma_0}$ and $L_1\in\Bl_{\sigma_1}$, such that $L=L_0\oplus L_1$. Therefore,

\begin{equation*}
\begin{split}
L = L_0\oplus L_1= &(U_1P_{\sigma_0}U_2)\oplus(U_3P_{\sigma_1}U_4) \\
 & = (U_1\oplus U_3)(P_{\sigma_0}U_2\oplus P_{\sigma_1}U_4) \\
 & = (U_1\oplus U_3)(P_{\sigma_0}\oplus P_{\sigma_1})(U_2\oplus U_4).
\end{split}
\end{equation*}
Since, $U_1\oplus U_3,\, U_2\oplus U_4\in\Up_{n+1}$, and $P_{\sigma_0}\oplus P_{\sigma_1}=P_{\sigma}$, then 
\[
L=\tilde{U}_1P_{\sigma}\tilde{U}_2,
\]
where $\tilde{U}_1=(U_1\oplus U_3)$ and $\tilde{U}_2=(U_2\oplus U_4)$.
Therefore, $L\in\Bl_{\sigma}$. 

In conclusion, $L\in\Bl_{\sigma}$ if and only if there exist $L_0\in\Bl_{\sigma_0}$ and $L_1\in\Bl_{\sigma_1}$, such that $L=L_0\oplus L_1$.
\end{proof}

\subsection{Split type 1}

Recall that $r_i$ is the horizontal row between the starting points of two adjacent wires at height $i+\frac{1}{2}$.

\begin{Def}
If a curve can be traced in the wiring diagram from $r_i$ 
to $r_{i+1}$, or from $r_{i+1}$ 
to $r_i$, such that it transversely crosses only one wire, without passing through an inversion, then a \textit{split type 1} can be performed on the diagram. This split is said to be performed at row $r_i$. The operation decomposes the diagram into two parts, resulting in permutations $\sigma_1\in\Sn_{i+1}$ and $\sigma_2\in\Sn_{n+1-i}$.  
\end{Def}

The permutations $\sigma_1 \in \Sn_{j+1}$ and $\sigma_2 \in \Sn_{n+1-j}$ 
are obtained by joining the wire that was cut with the dot
that does not have a wire entering or leaving it. 
It is important to note that the resulting words are still reduced.

\begin{Def}
Let $\sigma = a_{i_1} \dots a_{i_l}$ be a reduced word for a permutation $\sigma \in~ \Sn_{n+1}$. If a split type 1 can be performed at $r_j$, then:
\begin{itemize}
\renewcommand\labelitemi{$\bullet$}
    \item $\sigma_1=a_{i_{k_1}}\cdots a_{i_{k_{m}}}\in \Sn_{j+1}, \forall i_{k_s}\leq j$,
        \item $\sigma_2=a_{i_{k_1}-j}\cdots a_{i_{k_{n}}-j}\in \Sn_{n+1-j}, \quad \forall i_{k_s}> j$,
\end{itemize}
where $k_s\leq k_{s+1}$, $m=\inv(\sigma_1)$ and $n=\inv(\sigma_2)$.
\end{Def}

\begin{Ex}\label{exsplit1}
Let $\sigma=[325614]=a_1a_4a_3a_2a_1a_5a_4\in\Sn_6$.

We can trace a curve on the diagram, crossing the fifth wire and separating it into two parts.

 \begin{figure}[H]
    \centering
    \includegraphics[scale=0.07]{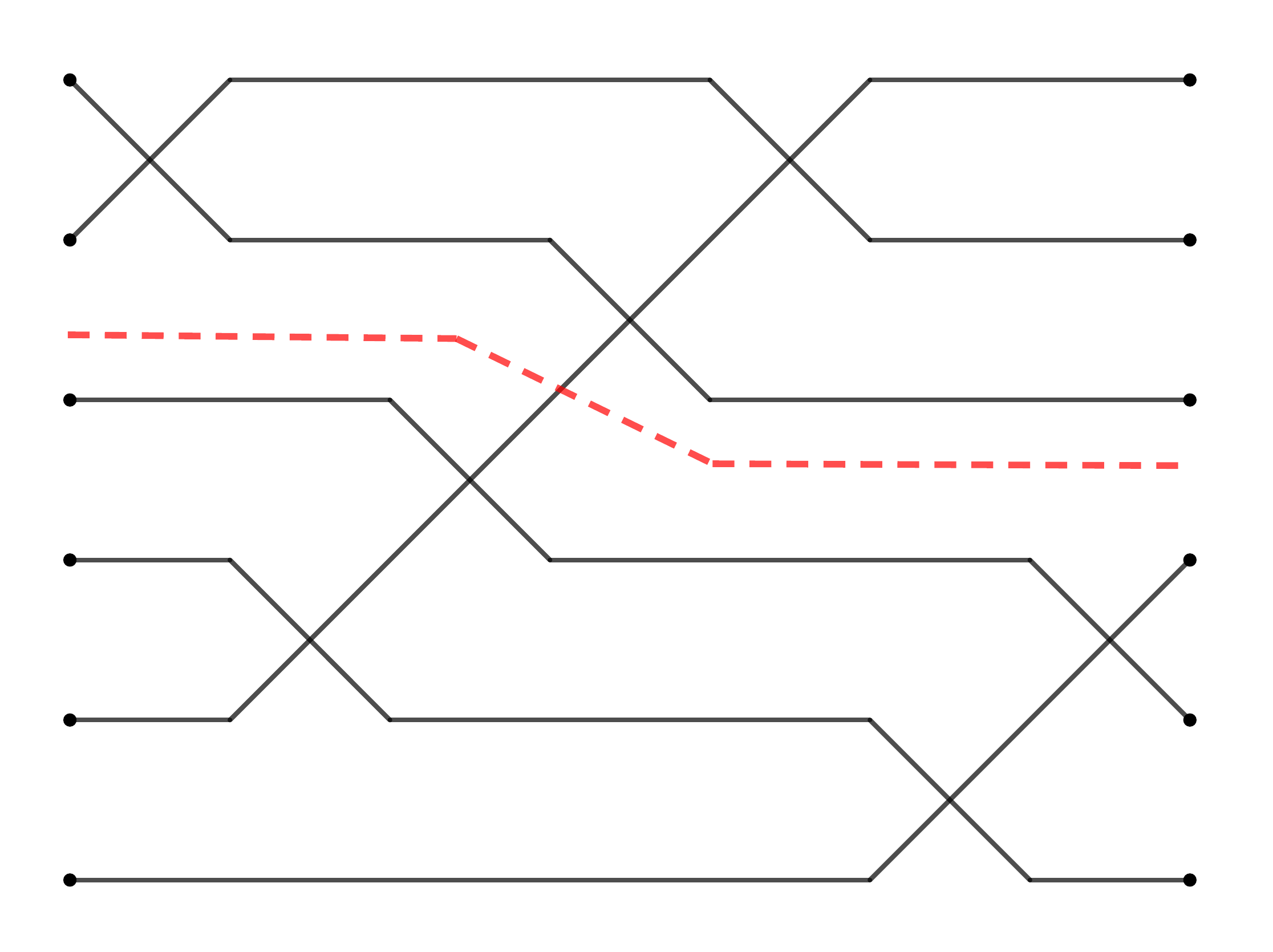}    
    \caption{First step to apply the split type 1 on the wiring diagram of the permutation $\sigma=a_1a_4a_3a_2a_1a_5a_4\in\Sn_6$.}
\end{figure}

The upper part is equivalent to $\sigma_{1}=a_1a_2a_1\in\Sn_3$, 
and the lower part is equivalent to $\sigma_{2}=a_2a_1a_3a_2\in\Sn_4$.
\end{Ex}
 \begin{figure}[H]
    \centering
    \includegraphics[scale=1.1]{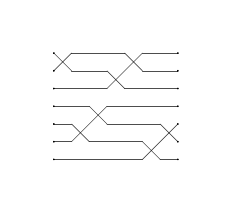}    
    \caption{Result of apply the split type 1.}
\end{figure}

\begin{Remark}\label{Afirm}
Let $\sigma = a_{i_1} \cdots a_{i_l} \in S_{n+1}$ be a reduced word, and let $\varepsilon_{0}$ be a preancestry for $\sigma$. Suppose that there exist indices $k_1 < k_2$ such that $\varepsilon_{0}(k_i) =~-2$, $i\in\{1,2\}$. If $|i_{k_1} - i_{k_2}| > 1$ then, for all $k$, the sequences
\[
\varepsilon_1(k) =
\begin{cases}
    0, & i_{k} \leq i_{k_1},\\
    \varepsilon_0(k), & i_{k} > i_{k_1},
\end{cases}
\quad \text{and} \quad
\varepsilon_2 (k)=
\begin{cases}
    0, & i_{k} \geq i_{k_2},\\
    \varepsilon_0(k), & i_{k} < i_{k_2},
\end{cases}
\]
are also preancestries for $\sigma$. 
This follows easily from the definition of a preancestry and we omit the details of the proof. This remark will be used to prove the next lemma.
\end{Remark}

\begin{Lem}\label{Lema split 1}
Let $\sigma = a_{i_1}\cdots a_{i_l}\in\Sn_{n+1}$ be a reduced word. 
If a split type 1 can be performed at $\sigma\in\Sn_{n+1}$, 
resulting in permutations $\sigma_1\in\Sn_{i+1}$ and $\sigma_2\in\Sn_{n+1-i}$, 
then $\Blc_{\sigma}=\Blc_{\sigma_1}\times \Blc_{\sigma_2}$.
\end{Lem}

\begin{proof}
When split type 1 is applied, the permutation is decomposed into two parts, $\sigma_1$ and $\sigma_2$, such that no region in one part has inversions lying on the boundary of a region in the other part.
Consequently, performing a click in the region of $\sigma$ corresponding to $\sigma_1$ does not affect the signs of the inversions associated with~$\sigma_2$. 
This establishes a correspondence between 1-skeletons of the desired CW-complexes. 

By the Remark \ref{Afirm}, we have a correspondence to higher dimensional cells. If $\varepsilon_0$ is a preancestry for $\sigma$ such that $i_{k_1} < r_i < i_{k_2}$, then $\varepsilon_1$ is a preancestry for $\sigma_1$, and $\varepsilon_2$ is a preancestry for ~$\sigma_2$.

This implies that the CW complex of $\sigma\in\Sn_{n+1}$ is equivalent to that of $\sigma_1\oplus\sigma_2\in\Sn_{n+2}$. 
By Lemma \ref{blocklemma}, we have $\Blc_{\sigma}=\Blc_{\sigma_1}\times \Blc_{\sigma_2}$.

\end{proof}

\subsection{Split type 2}

\begin{Def}
Let $\sigma = a_{i_1} \cdots a_{i_l}$ be a reduced word for a permutation $\sigma \in~\Sn_{n+1}$. A \textit{split type 2} on a wiring diagram at inversion $a_{i_k}$ is a decomposition of the diagram into two parts, which satisfies the following conditions: \begin{enumerate}
\item For all $j\neq k$, $a_{i_k}\neq a_{i_j}$;
\item
The remaining words in either $\Sn_{i_k+1}$ and $\Sn_{n+1-i_k}$, or $\Sn_{i_k}$ and $\Sn_{n+2-i_k}$, remain reduced.
\end{enumerate}

The inversion $a_{i_k}$ is called a \textbf{tourist}.
\end{Def}

Note that the \textbf{tourist} is an inversion that does not affect the component; 
it just gets affected. One could say that the inversion 
only observes what is happening, like a tourist.

The move involves separating the wiring diagram into two 
parts in such a way that the inversion $a_{i_k}$ is in one 
of the two parts, the wires that were cut are then 
reconnected to the dots that do not have a wire entering or leaving. 
The other part is obtained by connecting the wires that we have cut.

A split type 2 can be performed at $a_{i_k}$ by drawing a line at height $i_k+1/4$ or $i_k+3/4$. 
In the first case, the inversion will lie in the upper subdiagram; in the second case, it will lie in the lower subdiagram. In both scenarios, the split is said to occur at the inversion $a_{i_k}$.

When split type 2 is applied at $a_{i_k}$, at height $i_k + 1/4$, the resulting permutations are $\sigma_1\in\Sn_{i_k}$ and $\sigma_2\in\Sn_{n+2-i_k}$. In the other case, the resulting permutations are $\sigma_1\in\Sn_{i_k+1}$ and $\sigma_2\in\Sn_{n+1-i_k}$.

\begin{Remark}\label{remark-tourist}
If the tourist is not at the boundary of any region, split type 1 can be applied.   
\end{Remark}

\begin{Def}
Let $\sigma = a_{i_1} \cdots a_{i_l}$ be a reduced word for a permutation $\sigma \in~\Sn_{n+1}$. If a split type 2 can be performed at $a_j$, then:
\begin{enumerate}
    \item For $j+\frac{1}{4}$, the resulting permutations are:
    \begin{itemize}
\renewcommand\labelitemi{$\bullet$}
    \item $\sigma_1=a_{i_{k_1}}\cdots a_{i_{k_{m}}}\in\Sn_{j}, \quad \forall i_{k_s}\leq j-1$,
        \item $\sigma_2=a_{i_{k_1}-j-1}\cdots a_{i_{k_{n}}-j-1}\in\Sn_{n+2-j}, \quad \forall i_{k_s}> j-1$,
\end{itemize}
where $k_s\leq k_{s+1}$, $m=\inv(\sigma_1)$ and $n=\inv(\sigma_2)$.

    \item For $j+\frac{3}{4}$, the resulting permutations are:
    \begin{itemize}
\renewcommand\labelitemi{$\bullet$}
    \item $\sigma_1=a_{i_{k_1}}\cdots a_{i_{k_{m}}}\in\Sn_{j+1}, \quad \forall i_{k_s}\leq j,$
        \item $\sigma_2=a_{i_{k_1}-j}\cdots a_{i_{k_{n}}-j}\in\Sn_{n+1-j}, \quad \forall i_{k_s}> j$,
\end{itemize}
where $k_s\leq k_{s+1}$, $m=\inv(\sigma_1)$ and $n=\inv(\sigma_2)$.

\end{enumerate}
\end{Def}

\begin{Ex}\label{Ex split 2}
Let $\sigma=a_2a_1a_3a_2a_1a_5a_4\in\Sn_6$ be a reduced word. 
The inversion $a_3$ is a candidate for applying the split type 2. 
To begin, we mark the red line in Figure \ref{inv7sigma51-sum}, where the split type 2 is performed at the height $3 + \frac{3}{4}$.

    \begin{figure}[H]
    \centering
    \includegraphics[scale=0.08]{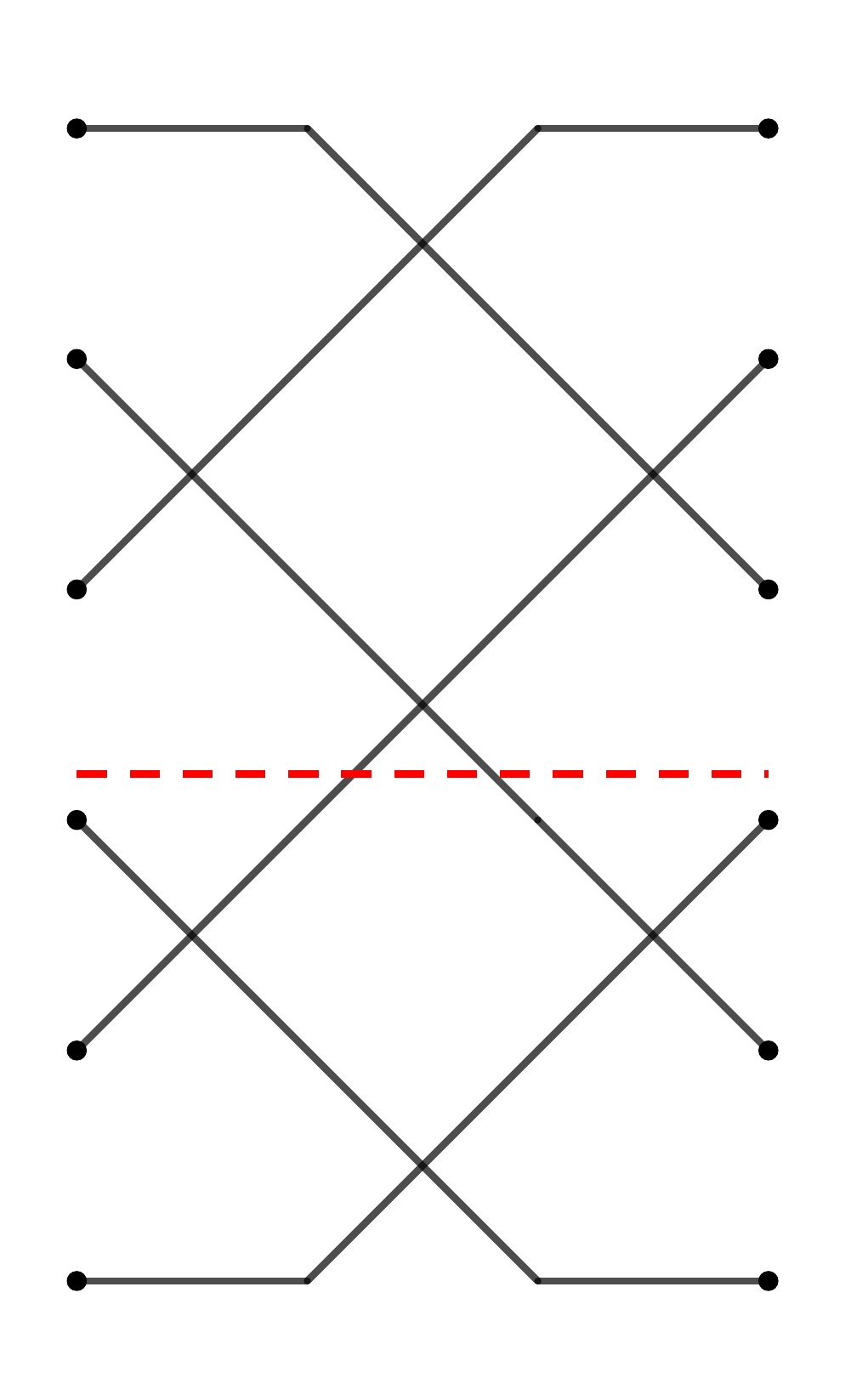}    
    \caption{First step to apply the split type 2 on $\sigma=a_2a_1a_4a_3a_2a_5a_4\in\Sn_6$.}
    \label{inv7sigma51-sum}
\end{figure}

Next, we connect the wires to form the resulting diagrams, as shown in Figure \ref{inv7sigma51-sum1}. Note that $a_2a_1a_3a_2\in\Sn_4$ and $a_1a_2a_1\in\Sn_3$ are still reduced. 

    \begin{figure}[H]
    \centering
    \includegraphics[scale=0.08]{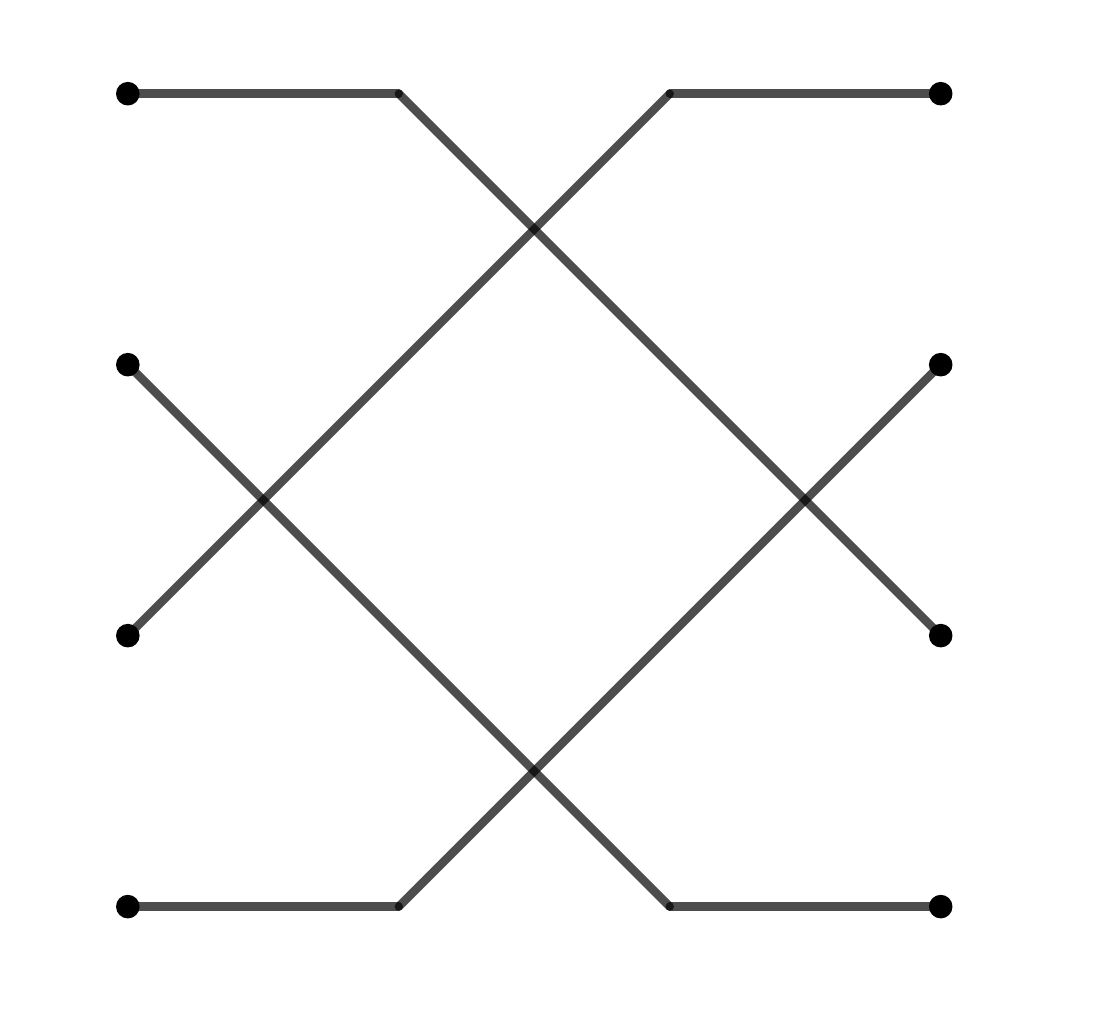}\quad   
    \includegraphics[scale=0.08]{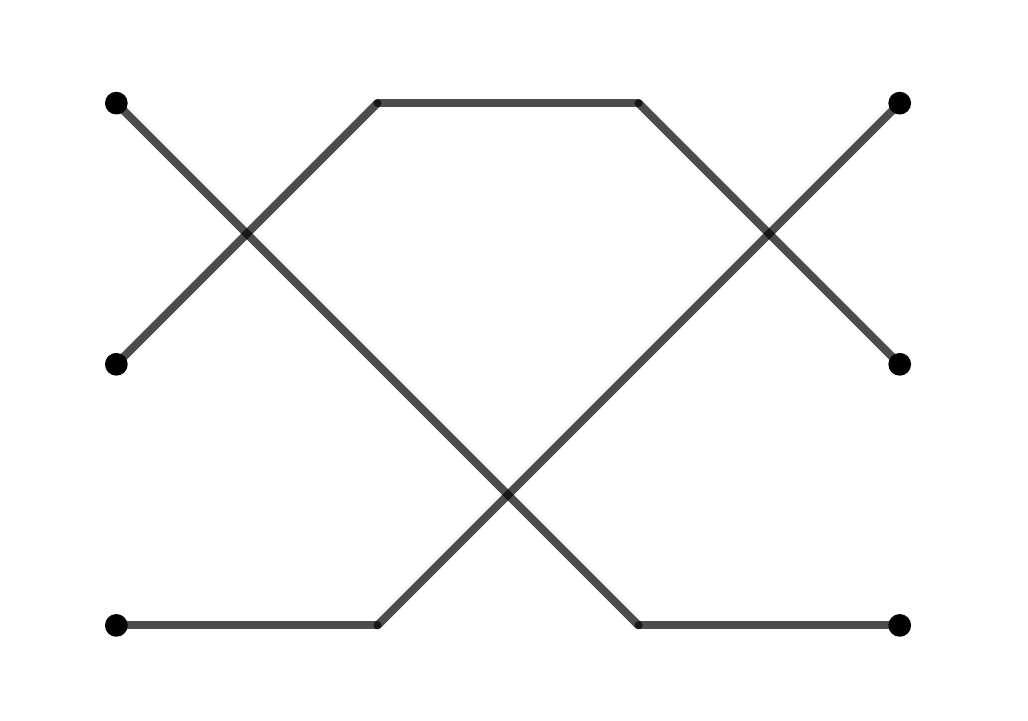}
    \caption{Resulting permutations: $\sigma_1=a_2a_1a_3a_2\in\Sn_4$ and $\sigma_2=a_1a_2a_1\in\Sn_3$.}
    \label{inv7sigma51-sum1}
\end{figure}

We can also apply split type 2 at the tourists $a_1, a_3,$ and $a_5$. 
\end{Ex}

\begin{Lem}\label{Lema split 2}
Let $\sigma = a_{i_1}\cdots a_{i_l}\in\Sn_{n+1}$ be a reduced word.
If $a_{i_k}$ is a tourist, then $\Blc_{\sigma}=\Blc_{\sigma_1}\times \Blc_{\sigma_2}$, 
where $\sigma_1\in\Sn_{{i_k}+1}$ and $\sigma_2\in\Sn_{{n+1}-{i_k}}$ 
are the remaining permutations obtained by performing a split type 2 at $a_{i_k}$.
\end{Lem}

The proof is similar to the proof of Lemma \ref{Lema split 1}.

\begin{proof}
Since $a_{i_k}$ is a tourist, there are no preancestries for $\sigma\in\Sn_{n+1}$ with the inversion $a_{i_k}$ marked. 
By applying split type 2, the permutation is decomposed into two parts such that row $r_{i_k}$ contains only the inversion $a_{i_k}$. Consequently, a preancestry for the original permutation corresponds to a pair of preancestries for $\sigma_1$ and $\sigma_2$.
This implies that the CW complex of $\sigma\in\Sn_{n+1}$ is equivalent to that of $\sigma_1\oplus\sigma_2\in\Sn_{n+2}$. 
By Lemma \ref{blocklemma}, we have $\Blc_{\sigma}=\Blc_{\sigma_1}\times \Blc_{\sigma_2}$.
\end{proof}

\subsection{Split Type 3}

Finally, we introduce the split type 3, which differs from the previous cases in that the decomposition is not a simple cut and its application produces an additional inversion.

\begin{Def}
Consider a wiring diagram where we trace a curve starting in $r_i$ at height $i+\frac{1}{2}$, that passes from $r_i$ to $r_{i-1}$ without crossing any wire. 
The curve then crosses a wire at height $i-\epsilon$, and moves up to height $i-\frac{1}{4}$. 
The curve moves horizontally at this height and then moves down, crossing another wire at height $i-\epsilon$. 
The curve then moves into $r_i$ and continues at height $i+\frac{1}{2}$ until the end. In the process the curve crosses wires exactly twice. 
We assume there are only two crossings in $r_{i-1}$. 
If such a curve can be traced in the wiring diagram, we can perform \textit{split type 3}. 
The operation decomposes the diagram into two parts, resulting in the permutations $\sigma_1\in\Sn_{i+1}$ and $\sigma_2\in\Sn_{n+2-i}$.
In the crossed region, the wires are first reconnected by linking the left wire to the right dot, and vice versa, creating an inversion $a_i$. 
The remaining wires are then connected by joining them at their nearest starting and ending points.
\end{Def}

\begin{Def}
Let $\sigma = a_{i_1} \dots a_{i_l}$ be a reduced word for a permutation $\sigma \in~\Sn_{n+1}$. If a split type 3 can be performed at $r_j$, then:
\begin{itemize}  \renewcommand\labelitemi{$\bullet$}
    \item $\sigma_1=a_{i_{k_1}}\cdots a_{i_{k_{m}}}\in\Sn_{j+1}, \quad \forall i_{k_s}\leq j$,
    \item $\sigma_2=a_{i_{k_1}-j-1}\cdots a_{i_{k_{n}}-j-1}\in\Sn_{n+2-j}, \quad \forall i_{k_s}\geq j$,
\end{itemize}
where $k_s\leq k_{s+1}$, $m=\inv(\sigma_1)$ and $n=\inv(\sigma_2)$. In $\sigma_1$, the subword $a_j\cdots a_j$ will be represented by a single $a_j$, which is the new inversion introduced by the split.
\end{Def}

\begin{Remark}
In split types 1 and 2, the permutation is decomposed into two smaller permutations whose dimensions sum to $n+2$.
In split type 3, however, one additional inversion is generated in $\sigma_1$. Here, the sum of the dimensions of $\sigma_1$ and $\sigma_2$ is $n+3$. The sign of the additional inversion does not alter the homotopy type of the associated CW complex. Simply taking the direct sum would result in twice as many components, so this must be adjusted accordingly.
\end{Remark}

\begin{Ex}
Let $\sigma=a_1a_2a_3a_2a_4a_3a_2a_5a_4a_3a_2a_1\in\Sn_6$.
In Figure \ref{directsum2}, we trace a red curve that only crosses one region in the diagram, in accordance with the conditions outlined in the definition.
\begin{figure}[H]
    \centering  
    \includegraphics[scale=1.1]{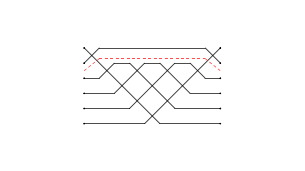}
    \caption{First step to perform a split type 3 on the diagram of $\sigma\in\Sn_6.$}
    \label{directsum2}
\end{figure}

Now, we connect the wires that we cut in the upper part of the diagram
to the dots representing 3 on both sides, 
creating an inversion $a_3$ in the diagram. 
The resulting permutation is $\eta \in \Sn_3$.

After that, we connect the wires that we cut in the lower
part of the wiring diagram to the dots representing 1 on both sides.
The resulting permutation is $\eta\in\Sn_5$.
\end{Ex}

\begin{figure}[H]
    \centering  
    \includegraphics[scale=0.9]{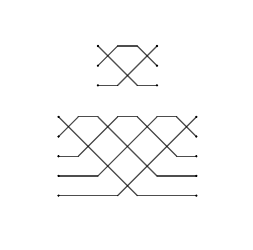}
    \caption{Resulting permutations $\eta\in\Sn_3$ and $\eta\in\Sn_5.$}
    \label{decom}
\end{figure}

\begin{Remark}\label{Afirm2}
Let $\sigma = a_{i_1} \cdots a_{i_l} \in S_{n+1}$ be a reduced word, and let $\varepsilon_{0}$ be a preancestry for $\sigma$. Suppose that there exist indices $k_1 < k_2$ such that $\varepsilon_{0}(k_i) =~-2$, $i\in\{1,2\}$. If $|i_{k_1} - i_{k_2}| = 1$, and exist a unique $k$ such that $i_k=i_{k_1}$ then, for all $k$, the sequences
\[
\varepsilon_1(k) =
\begin{cases}
    0, & i_{k} \leq i_{k_1},\\
    \varepsilon_0(k), & i_{k} > i_{k_1},
\end{cases}
\quad \text{and} \quad
\varepsilon_2 (k)=
\begin{cases}
    0, & i_{k} \geq i_{k_2},\\
    \varepsilon_0(k), & i_{k} < i_{k_2},
\end{cases}
\]
are also preancestries for $\sigma$. This follows easily from the definition of a preancestry and we omit the details of the proof. This remark will be used to prove the next lemma.
\end{Remark}

\begin{Lem} \label{Lemma split 3}
Let $\sigma = a_{i_1}\cdots a_{i_l}\in\Sn_{n+1}$ be a reduced word. 
If a split type 3 can be performed at 
$\sigma\in\Sn_{n+1}$, resulting in $\sigma_1\in\Sn_{j+2}$ 
and $\sigma_2\in\Sn_{n+1-j}$, 
then $\Blc_{\sigma}\times\{\pm1\}=\Blc_{\sigma_1}\times \Blc_{\sigma_2}$. In particular, the number of connected components in $\Blc_{\sigma}$ 
is half the product of the number of connected components 
in $\Blc_{\sigma_1}$ and $\Blc_{\sigma_2}.$
\end{Lem}

\begin{proof}
Applying split type 3 the permutation is decomposed into two parts: the upper part, which includes the new inversion $a_j$ and is represented by $\sigma_1 \in \Sn_{j+2}$, and the lower part. 
Performing a click operation in the region corresponding to $\sigma_1$ changes the signs of all inversions in $r_j$ simultaneously. 
Consequently, this change of signs does not affect the possibility of performing a click in the regions corresponding to $\sigma_2$.
This establishes a correspondence between 1-skeletons of the desired CW-complexes. 

The curve passes through exactly one region which is contained between the only two crossings in $r_{i-1}$. Thus, any preancestry with squares in $r_{i-1}$ has only one possible way to be marked in this row. 
By Remark \ref{Afirm2}, we have a correspondence to higher dimensional cells. Therefore, if $\varepsilon_0$ is a preancestry for $\sigma$ such that $r_{i-1}=k_1$, then $\varepsilon_1$ and $\varepsilon_2$ correspond to a pair of preancestries in $\sigma_1$ and $\sigma_2$, respectively. 

If no click is performed in the region corresponding to $\sigma_1$, the sign of $a_j$ in $\sigma_1$ is $\circ$ (or $\sbc$); if a click is performed, the sign changes to $\sbc$ (or $\circ$). This results in two copies of the same component.

Therefore, $\Blc_{\sigma}\times\{\pm1\}=\Blc_{\sigma_0}$, where $\sigma_0=\sigma_1\oplus\sigma_2\in\Sn_{n+3}$. By Lemma \ref{blocklemma}, it follows that $\Blc_{\sigma_0} = \Blc_{\sigma_1} \times \Blc_{\sigma_2}$. Thus, the number of connected components of $\Blc_{\sigma}$ is half the product of the number of connected components of $\Blc_{\sigma_1}$ and $\Blc_{\sigma_2}$.
\end{proof}

\begin{Ex}
Let $\sigma=a_1a_2a_3a_2a_4a_3a_2a_5a_4a_3a_2a_1\in\Sn_6$, as in the previous example.
In the diagram of $\sigma_1=a_1a_2a_1\in\Sn_3$ (Figure \ref{decom}),
the inversion $a_2$ is generated during the process of separating the diagrams. However, this inversion does not affect the overall analysis. Its effect is limited to changing the signs of
$r_1$ in the diagram of $\sigma_2=a_1a_2a_1a_3a_2a_1a_4a_3a_2a_1\in\Sn_5$.

The CW complex is formed as the product of $\Blc_{\sigma_2}$ 
and the cells of $\Blc_{\sigma_1}$ disregarding the influence of $a_2$. 
In this example, the cells correspond to two dots and one segment, represented as $(\sbc\, x\, \sbc), (\circ\, x\, \circ),$ and $(\sbd\, x\, \sd)$. 
Assigning either $\circ$ or $\sbc$ to the position $x$ yields the same CW complex. 

Therefore, only one possibility needs to be considered. Consequently, the number of connected components is half the product of the components.

From this analysis, it follows that $\Blc_{\sigma}$ 
contains $3\times 52=156$ connected components, all of which are contractible.
\end{Ex}

In the following sections, we examine the homotopy type of $\Bl_{\sigma}$ with $\sigma\in\Sn_6$ categorizing the analysis by the number of inversions.
With our understanding of how to decompose a wiring diagram, we can now distinguish between permutations that can be reduced and those that cannot.
\section{The Homotopy Type of $\Bl_{\sigma}$ for $\inv(\sigma)\leq 8$}\label{inv6}

The decomposition techniques introduced in this work, particularly the split operations, provide a straightforward method to determine the homotopy type of $\sigma \in \Sn_6$, with $\inv(\sigma) \leq 8$.

Here and in the following sections, we present information on the homotopy type by listing permutations according to the presence of blocks, and indicating the cases in which a particular type of split can be applied. When no type of split can be applied, we proceed to construct the components.

\subsection{Permutations in $\Sn_6$ with $\inv(\sigma)=4$}

Permutations $\sigma\in\Sn_6$ with $\inv(\sigma)\leq 4$ have a non-trivial block structure, that is, $\block(\sigma) \neq 0$. This allows them to be expressed as a sum of well-known permutations.

If a permutation $\sigma$ decomposes as a direct sum $\sigma=\sigma_1\oplus\sigma_2$, then the corresponding set $\Bl_{\sigma}$ also decomposes as $\Bl_{\sigma}=\Bl_{\sigma_1}\oplus\Bl_{\sigma_2}$. Since $\sigma_1\in\Sn_j$ and $\sigma_2\in\Sn_{6-j}$ with $j\leq5$, both $\Bl_{\sigma_1}$ and $\Bl_{\sigma_2}$ are contractible. Consequently, the sum $\Bl_{\sigma}=\Bl_{\sigma_1}\oplus\Bl_{\sigma_2}$ is also contractible.
This decomposition also implies that the number of connected components of $\Bl_{\sigma}$ is the product of the number of connected components of $\Bl_{\sigma_1}$ and $\Bl_{\sigma_2}$.

We will now categorize the permutations according to their number of blocks. For any permutation $\sigma\in\Sn_{n+1}$ with $b\neq0$, the homotopy type of $\Bl_{\sigma}$ is trivial.

\subsection{Permutations in $\Sn_6$ with $\inv(\sigma)=5$} 
For $\inv(\sigma)=5$ there are 71 permutations distributed across the following cases:

\begin{enumerate}
\item There are 55 permutations with $b\neq 0$;

\item There are 16 permutations with $b=0$.  

In this case, we can apply split type 1. Consequently, $\Bl_{\sigma}$ is contractible.
    
\end{enumerate}

Since $\Bl_{\sigma}$ is contractible for both $b=0$ and $b\neq 0$, it follows that $\Bl_{\sigma}$ is contractible for all $\sigma \in \Sn_6$ with $\inv(\sigma)= 5$.

\subsection{Permutations in $\Sn_6$ with $\inv(\sigma)=6$}
For $\inv(\sigma)=6$, there are 90 permutations distributed across the following cases:

\begin{enumerate}
\item There are 46 permutations with $b\neq0$;

\item There are 44 permutations with $b=0$;

In this case, we can apply split type 1. Consequently, $\Bl_{\sigma}$ is contractible.

\end{enumerate}

As a result, for all $\sigma \in \Sn_6$ with $\inv(\sigma)=6$, $\Bl_{\sigma}$ is contractible.

\subsection{Permutations in $\Sn_6$ with $\inv(\sigma)=7$}
For $\inv(\sigma)=7$, there are 101 permutations distributed across the following cases:

\begin{enumerate}
    \item There are 32 permutations with $b\neq 0$; 

    \item There are 68 permutations that we can apply split type 1. Consequently, $\Bl_{\sigma}$ is contractible.

\item For the permutation $\sigma=a_2a_1a_4a_3a_2a_5a_4$ we can apply split type 3. Therefore, $\Bl_{\sigma}$ is contractible.

\end{enumerate}

As a result, for all $\sigma \in \Sn_6$ with $\inv(\sigma)=7$, $\Bl_{\sigma}$ is contractible.

\subsection{Permutations in $\Sn_6$ with $\inv(\sigma)=8$} 
For $\inv(\sigma)=8$, there are 101 permutations distributed across the following cases: 

\begin{enumerate}
    \item There are 18 permutations with $b\neq 0$;    
    \item There are 56 permutations that we can apply split type 1;
    
    \item There are 18 permutations that we can apply split type 1 or 2;

    \item There are 9 permutations that we can apply split type 2;
\end{enumerate}

Therefore, for all cases, $\Bl_{\sigma}$ is contractible.
In conclusion, $\Bl_{\sigma}$ is contractible for all $\sigma \in \Sn_6$ with $\inv(\sigma) \leq 8.$

\section{The Homotopy Type of $\Bl_{\sigma}$ for $9\leq\inv(\sigma)\leq12$}\label{inv12}

For permutations $\sigma\in\Sn_6$ with $9\leq\inv(\sigma)\leq12$, the split techniques apply to some cases, while others require individual analysis. 
For the latter, we produced explicit drawings and examined their connected components directly. 
We present one representative case here; for the remaining diagrams, see \cite{leal2025homotopy}.

\subsection{Permutations in $\Sn_6$ with $\inv(\sigma)=9$}
For $\inv(\sigma)=9$, we have 90 permutations distributed in the following cases: 

\begin{enumerate}
    \item There are 8 permutations with $b\neq 0$;
    
    \item There are 40 permutations that we can apply split type 1;
    
    \item There are 15 permutations that we can apply split type 2;
    
    \item There are 19 permutations that we can apply split type 3;
    
    \item There are 8 permutations that need to be studied separately.
\end{enumerate}

\subsubsection{Case 5}\label{case5}

 For $\sigma=[651234]=a_4a_3a_2a_1a_5a_4a_3a_2a_1\in\Sn_6$, it follows that
    \begin{equation*}
    \begin{aligned}
\acute{\sigma}&=\frac{1}{4\sqrt{2}}(-1-\hat{a}_1-\hat{a}_2-\hat{a}_1\hat{a}_2-\hat{a}_3+\hat{a}_1\hat{a}_3-\hat{a}_2\hat{a}_3+\hat{a}_1\hat{a}_2\hat{a}_3-\hat{a}_4+\hat{a}_1\hat{a}_4\\ &+\hat{a}_2\hat{a}_4-\hat{a}_1\hat{a}_2\hat{a}_4-\hat{a}_3\hat{a}_4-\hat{a}_1\hat{a}_3\hat{a}_4\hat{a}_2\hat{a}_3\hat{a}_4+\hat{a}_1\hat{a}_2\hat{a}_3\hat{a}_4-\hat{a}_5+\hat{a}_1\hat{a}_5+\hat{a}_2\hat{a}_5\\ & -\hat{a}_1\hat{a}_2\hat{a}_5+\hat{a}_3\hat{a}_5+\hat{a}_1\hat{a}_3\hat{a}_5-\hat{a}_2\hat{a}_3\hat{a}_5-\hat{a}_1\hat{a}_2\hat{a}_3\hat{a}_5-\hat{a}_4\hat{a}_5-\hat{a}_1\hat{a}_4\hat{a}_5-\hat{a}_2\hat{a}_4\hat{a}_5\\ &-\hat{a}_1\hat{a}_2\hat{a}_4\hat{a}_5+\hat{a}_3\hat{a}_4\hat{a}_5-\hat{a}_1\hat{a}_3\hat{a}_4\hat{a}_5+\hat{a}_2\hat{a}_3\hat{a}_4\hat{a}_5-\hat{a}_1\hat{a}_2\hat{a}_3\hat{a}_4\hat{a}_5).      
    \end{aligned}    
    \end{equation*}
    
There exist $2^5=32$ thin ancestries. 
Consequently, $\Bl_{\sigma}$ has 32 thin connected components, all contractible.

For ancestries of dimension $1$, there are four possible positions for the squares:

\begin{enumerate}[label=(\roman*)]
    \item the rows containing $(\sbc,\circ)$ or $(\circ,\sbc)$ are consecutive, and there are exactly two such rows; or
    \item exactly one of the rows $r_1$ and $r_4$ contains $(\sbc,\circ)$ or $(\circ,\sbc)$, while all remaining rows have the same sign at every inversion.
\end{enumerate}

In each case, the resulting CW complex is the one shown in Figure~\ref{inv9sigma15-2}. Consequently, we obtain $32$ copies of this complex.
\begin{figure}[H]
   \centerline{
    \includegraphics[width=1.0\textwidth]{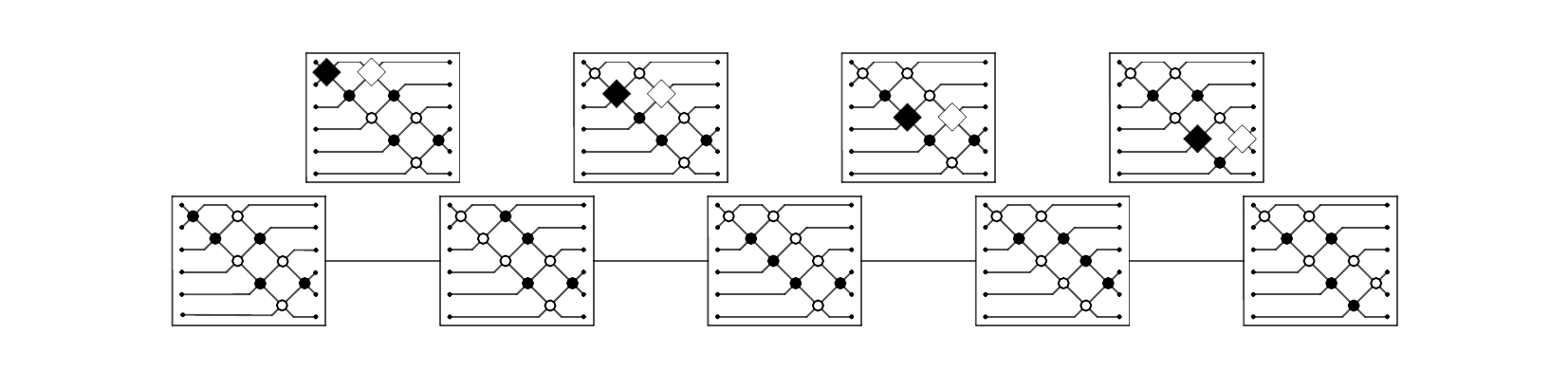}}
    \caption{CW complex of dimension 1 generated by $\varepsilon_1=(\sbd\sbc\sd\circ\sbc\sbc\circ\circ\sbc)$.}
    \label{inv9sigma15-2}
\end{figure}

Therefore, $\Bl_{\sigma}$ has a total of 32 connected components of this type, all contractible.

The remaining ancestries of dimension 1 appear in the 2-dimensional CW complex.

For dimension 2, there are three possible positions for the squares that will appear together. 
This yields 32 copies of the CW complex shown in Figure \ref{inv9sigma15-21}.

\begin{figure}[H]
   \centerline{
    \includegraphics[width=0.75\textwidth]{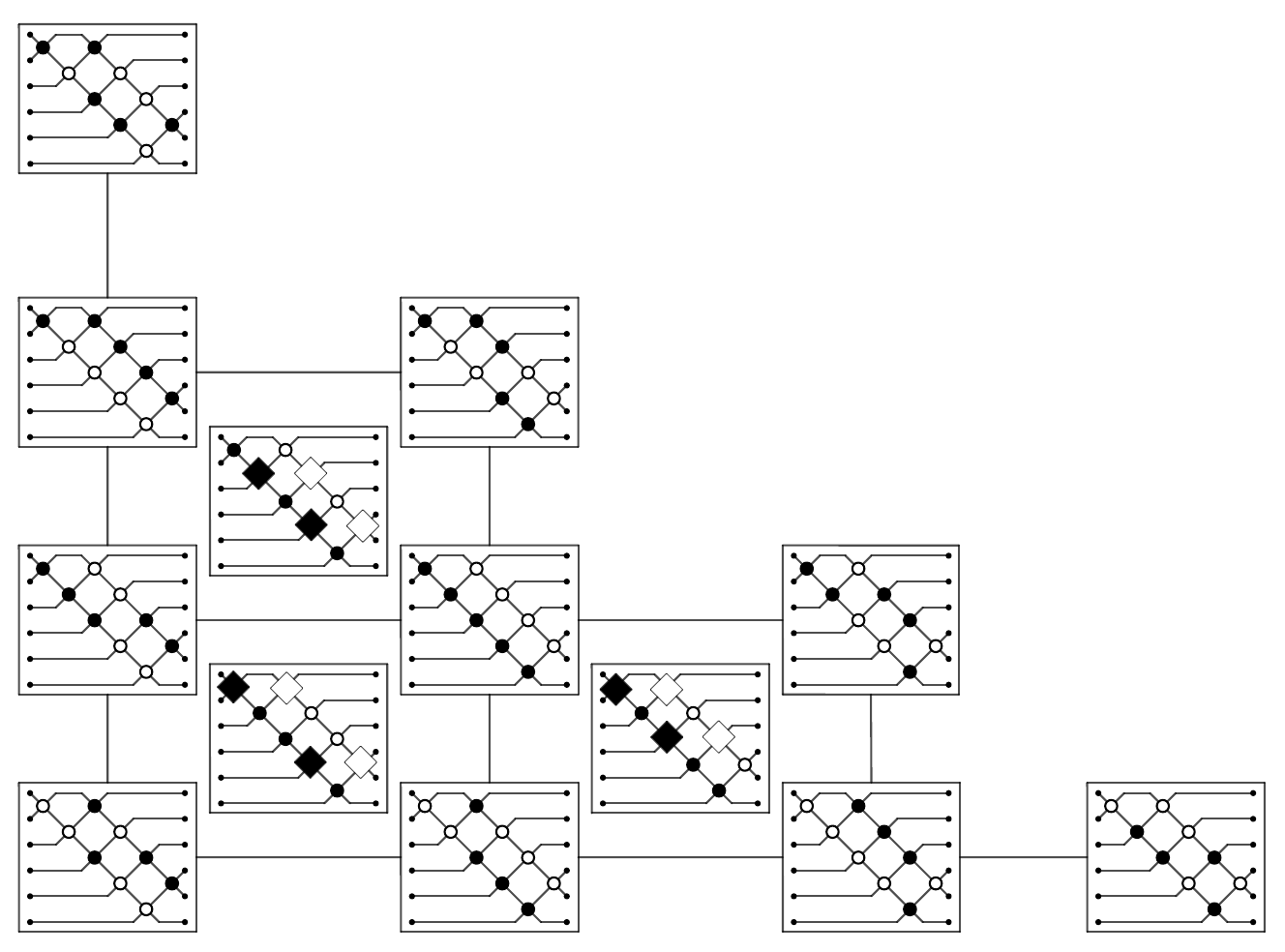}}
    \caption{CW complex of dimension 2 with ancestries $\varepsilon_2=(\sbd\sbc\sd\sbd\circ\sbc\sd\sbc\circ)$,\quad$\varepsilon_3=(\sbd\sbc\sd\sbc\circ\sbd\circ\sbc\sd)$ and $\varepsilon_4=(\sbc\circ\sbd\sbc\sd\sbd\circ\sbc\sd)$.}
    \label{inv9sigma15-21}
\end{figure}

Therefore, $\Bl_{\sigma}$ has a total of 32 connected components of this type, all contractible.
In summary, $\Bl_{\sigma}$ has 96 connected components, all of which are contractible.

As a result, $\Bl_{\sigma}$ is contractible for all $\sigma\in\Sn_6$ with $\inv(\sigma)=9$.

\subsection{Permutations in $\Sn_6$ with $\inv(\sigma)=10$} 
For $\inv(\sigma)=10$, we have 71 permutations distributed in the following cases:

\begin{enumerate}
    \item There are 2 permutations with $b\neq0$;
    \item There are 16 permutations that we can apply split type 1;
    \item There are 12 permutations that we can apply split type 2;

    \item There are 21 permutations that we can apply split type 3;
    
    \item There are 20 permutations that need to be studied separately.
\end{enumerate}

Following a method similar to the one in subsection \ref{case5}, we analyze the final case and find that all components are contractible. The illustrations supporting this analysis can be found in \cite{leal2025homotopy}.

As a result, $\Bl_{\sigma}$ is contractible for all $\sigma\in\Sn_6$ with $\inv(\sigma)=10$.

\subsection{Permutations in $\Sn_6$ with $\inv(\sigma)=11$}
For $\inv(\sigma)=11$, we have 49 permutations distributed in the following cases:

\begin{enumerate}
    \item There are 4 permutations whose analysis can be reduced to the case of the permutation $\sigma_1=a_1a_2a_1a_3a_2a_1a_4a_3a_2a_1\in\Sn_5$.  In these cases, we apply a split of type 1 or type 2.;
    
    \item There are 3 permutations that we can apply split type 2;

    \item There are 9 permutations that we can apply split type 3;

    \item There are 33 permutations that need to be studied separately.
\end{enumerate}

In the final case, we illustrate the components and show that they are all contractible. A visual representation of these components can be found in \cite{leal2025homotopy}.

Therefore, $\Bl_{\sigma}$ is contractible for all $\sigma\in\Sn_6$ with $\inv(\sigma)=11$.

\subsection{Permutations in $\Sn_6$ with $\inv(\sigma)=12$}
For $\inv(\sigma)=12$, except for $\sigma=[563412]$, there are 28 permutations distributed in two cases:

\begin{enumerate}
    \item There are 2 permutations that we can apply split type 3;
    \item There are 26 permutations that need to be studied separately.
\end{enumerate}

For the last case, we illustrate the components and conclude that all are contractible. For these drawings, refer to \cite{leal2025homotopy}.

We conclude that $\Bl_{\sigma}$ is contractible for all $\sigma\in\Sn_6$ with $\inv(\sigma)=12$, with the single exception of $\sigma=[563412]$.

\section{Permutation $\sigma=[563412]$}\label{noncontractible}

For $\sigma=[563412]= a_2a_1a_3a_2a_4a_3a_2a_1a_5a_4a_3a_2 \in \Sn_6$, it follows that 

\[
\acute{\sigma}=\frac{1}{2}(-\hat{a}_1-\hat{a}_2\hat{a}_3\hat{a}_4-\hat{a}_5+\hat{a}_1\hat{a}_2\hat{a}_3\hat{a}_4\hat{a}_5)\in\tilde{\Bn}^{+}_6.
\]

In the next section, we explore the orbits of the elements $z\in\acute{\sigma}\Qt_6$, as well as the count of cells of each dimension present in the component.

\subsection{The Orbits}

The set $\acute{\sigma}\Qt_6$ consists of nine orbits each of size 4 or 8:
\[
\mathcal{O}_{\acute{\sigma}}=\Bigl\{\frac{\pm\hat{a}_1\pm\hat{a}_2\hat{a}_3\hat{a}_4\pm\hat{a}_5\pm\hat{a}_1\hat{a}_2\hat{a}_3\hat{a}_4\hat{a}_5}{2}\Bigl\},
\]
\[
\mathcal{O}_{\hat{a}_1\acute{\sigma}}=\Bigl\{\frac{1\pm\hat{a}_1\hat{a}_2\hat{a}_3\hat{a}_4\pm\hat{a}_1\hat{a}_5\pm\hat{a}_2\hat{a}_3\hat{a}_4\hat{a}_5}{2}\Bigl\},
\]
\[
\mathcal{O}_{\hat{a}_2\acute{\sigma}}=\Bigl\{\frac{\pm\hat{a}_1\hat{a}_2\pm\hat{a}_3\hat{a}_4\pm\hat{a}_2\hat{a}_5\pm\hat{a}_1\hat{a}_3\hat{a}_4\hat{a}_5}{2}\Bigl\},
\]
\[
\mathcal{O}_{\hat{a}_1\hat{a}_2\acute{\sigma}}=\Bigl\{\frac{\pm\hat{a}_2\pm\hat{a}_1\hat{a}_3\hat{a}_4\pm\hat{a}_1\hat{a}_2\hat{a}_5\pm\hat{a}_3\hat{a}_4\hat{a}_5}{2}\Bigl\},
\]
\[
\mathcal{O}_{\hat{a}_3\acute{\sigma}}=\Bigl\{\frac{\pm\hat{a}_1\hat{a}_3\pm\hat{a}_2\hat{a}_4\pm\hat{a}_3\hat{a}_5\pm\hat{a}_1\hat{a}_2\hat{a}_4\hat{a}_5}{2}\Bigl\},
\]
\[
\mathcal{O}_{\hat{a}_1\hat{a}_3\acute{\sigma}}=\Bigl\{\frac{\pm\hat{a}_3\pm\hat{a}_1\hat{a}_2\hat{a}_4\pm\hat{a}_1\hat{a}_3\hat{a}_5\pm\hat{a}_2\hat{a}_4\hat{a}_5}{2}\Bigl\},
\]
\[
\mathcal{O}_{\hat{a}_2\hat{a}_3\acute{\sigma}}=\Bigl\{\frac{\pm\hat{a}_1\hat{a}_2\hat{a}_3\pm\hat{a}_4\pm\hat{a}_1\hat{a}_4\hat{a}_5\pm\hat{a}_2\hat{a}_3\hat{a}_5}{2}\Bigl\},
\]
\[
\mathcal{O}_{\hat{a}_4\acute{\sigma}}=\Bigl\{\frac{\pm\hat{a}_2\hat{a}_3\pm\hat{a}_1\hat{a}_4\pm\hat{a}_4\hat{a}_5\pm\hat{a}_1\hat{a}_2\hat{a}_3\hat{a}_5}{2}\Bigl\},
\]
\[
\mathcal{O}_{-\hat{a}_1\acute{\sigma}}=\Bigl\{\frac{-1\pm\hat{a}_1\hat{a}_2\hat{a}_3\hat{a}_4\pm\hat{a}_1\hat{a}_5\pm\hat{a}_2\hat{a}_3\hat{a}_4\hat{a}_5}{2}\Bigl\}.
\]

In the expressions within the Clifford algebra notation, 
the signs must be such that there is an odd number of equal signs.

The elements $z\in\acute{\sigma}\Qt_6$ have $\mathfrak{R}(z)\in\{-\frac{1}{2},0,\frac{1}{2}\}$.
Using the Formula \ref{sumdim0} of the number of ancestries of
dimension 0 for a given $z\in\acute{\sigma}\Qt_6$, it follows that $\Na(z)\in\{48,64,80\}$.
The number of ancestries per dimension can be determined using the Formulas \ref{sumneg} and \ref{sumpos}, and this can be cross-verified using Maple.

\begin{enumerate}
    \item For $z\in\mathcal{O}_{\acute{\sigma}}$, $\mathfrak{R}(z)=0$ and $\Na(z)=64$ and $\Na_{thin}(z)=4$. For each $z\in\mathcal{O}_{\acute{\sigma}}$, the sets $\Bl_{z}$ have four thin components and one thick. Consequently, $\Bl_{\sigma}$ has 40 connected components of this type.     

   The component has sixty 0-cells, one hundred and twelve 1-cells, sixty-eight 2-cells, sixteen 3-cells, and one 4-cell. Moreover, the Euler characteristic of this component is equal to 1.
    
    \item If $\mathfrak{R}(z)=\frac{1}{2}$, then $\Na(z)=80$ and $\Na_{thin}(z)=0$. For each $z\in\mathcal{O}_{\hat{a}_1\acute{\sigma}}$, the sets $\Bl_{z}$ have one connected component. Therefore, $\Bl_{\sigma}$ has 4 connected components of this type.

    The component has eighty 0-cells, one hundred and sixty-eight 1-cells, one hundred and twenty-eight 2-cells, forty-eight 3-cells, ten 4-cells, and one 5-cell. Additionally, the Euler characteristic of this component is 1.
    
    \item For $z\in\mathcal{O}_{\hat{a}_2\acute{\sigma}}$, $\mathfrak{R}(z)=0$ and $\Na(z)=64$ and $\Na_{thin}(z)=0$. For each $z\in\mathcal{O}_{\hat{a}_2\acute{\sigma}}$, the sets $\Bl_{z}$ have one connected component.
    Consequently, $\Bl_{\sigma}$ has 8 connected components of this type.
    
    The component has sixty-four 0-cells, one hundred and twelve 1-cells, sixty 2-cells, twelve 3-cells, and one 4-cell. In this case, the Euler characteristic is also equal to 1.
    
    \item For $z\in\mathcal{O}_{\hat{a}_1\hat{a}_2\acute{\sigma}}$, $\mathfrak{R}(z)=0$ and $\Na(z)=64$ and $\Na_{thin}(z)=0$. For each $z\in\mathcal{O}_{\hat{a}_1\hat{a}_2\acute{\sigma}}$, the sets $\Bl_{z}$ have one connected component. Therefore, $\Bl_{\sigma}$ has 8 connected components of this type. 
    
    The component has sixty-four 0-cells, one hundred and twelve 1-cells, sixty 2-cells, twelve 3-cells, and one 4-cell. Furthermore, this component has an Euler characteristic of 1.
    
    \item For $z\in\mathcal{O}_{\hat{a}_3\acute{\sigma}}$, $\mathfrak{R}(z)=0$ and $\Na(z)=64$ and $\Na_{thin}(z)=0$. For each $z\in\mathcal{O}_{\hat{a}_3\acute{\sigma}}$, the sets $\Bl_{z}$ have one connected component. Hence, $\Bl_{\sigma}$ has 8 connected components of this type.

The component has sixty-four 0-cells, one hundred and twelve 1-cells, fifty-two 2-cells, and four 3-cells. Moreover, the Euler characteristic of this component is 0.
    
    \item For $z\in\mathcal{O}_{\hat{a}_1\hat{a}_3\acute{\sigma}}$, $\mathfrak{R}(z)=0$ and $\Na(z)=64$ and $\Na_{thin}(z)=0$. For each $z\in\mathcal{O}_{\hat{a}_1\hat{a}_3\acute{\sigma}}$, the sets $\Bl_{z}$ have one connected component. Therefore, $\Bl_{\sigma}$ has 8 connected components of this type.

The component has sixty-four 0-cells, one hundred and twelve 1-cells, fifty-two 2-cells, and four 3-cells. In addition, the Euler characteristic of this component equals 0.

    \item For $z\in\mathcal{O}_{\hat{a}_2\hat{a}_3\acute{\sigma}}$, $\mathfrak{R}(z)=0$ and $\Na(z)=64$ and $\Na_{thin}(z)=0$. For each $z\in\mathcal{O}_{\hat{a}_2\hat{a}_3\acute{\sigma}}$, the sets $\Bl_{z}$ have one connected component. Consequently, $\Bl_{\sigma}$ has 8 connected components of this type.

    The component has sixty-four 0-cells, one hundred and twelve 1-cells, sixty 2-cells, twelve 3-cells and one 4-cell. Moreover, the Euler characteristic of this component is equal to 1.
    
    \item For $z\in\mathcal{O}_{\hat{a}_4\acute{\sigma}}$, $\mathfrak{R}(z)=0$ and $\Na(z)=64$ and $\Na_{thin}(z)=0$. For each $z\in\mathcal{O}_{\hat{a}_4\acute{\sigma}}$, the sets $\Bl_{z}$ have one connected component. Therefore, $\Bl_{\sigma}$ has 8 connected components of this type. 

    The component has sixty-four 0-cells, one hundred and twelve 1-cells, sixty 2-cells, twelve 3-cells and one 4-cell. Furthermore, this component has an Euler characteristic of 1.
    
    \item If $\mathfrak{R}(z)=-\frac{1}{2}$, then $\Na(z)=48$ and $\Na_{thin}(z)=0$. For each $z\in\mathcal{O}_{-\hat{a}_1\acute{\sigma}}$, the sets $\Bl_{z}$ have two connected components. Hence, $\Bl_{\sigma}$ has 8 connected components of this type. 

    The component has twenty-four 0-cells, twenty-eight 1-cells and four 2-cells. Furthemore, the Euler characteristic of this component is 0.
    
\end{enumerate}

\subsection{The Known Component}\label{dim2}

In \cite{alves2022onthehomotopy}, Alves and Saldanha presented a connected component of $\Bl_{\sigma}$ that is homotopically equivalent to $\mathbb{S}^1$ and thus non-contractible. In this section, we will briefly recall this component, which corresponds to the orbit $\mathcal{O}_{-\hat{a}_1\acute{\sigma}}$.

If $r_4$ has opposite signs and the remaining rows have equal signs, we obtain the component shown in Figure \ref{inv12sigma12-2} and corresponds to the CW complex depicted in Figure \ref{inv12sigma12-2cw}.
Therefore, $\Bl_{\sigma}$ has 8 connected components of this type.

\begin{figure}[H]
   \centerline{    
   \includegraphics[width=0.6\textwidth]{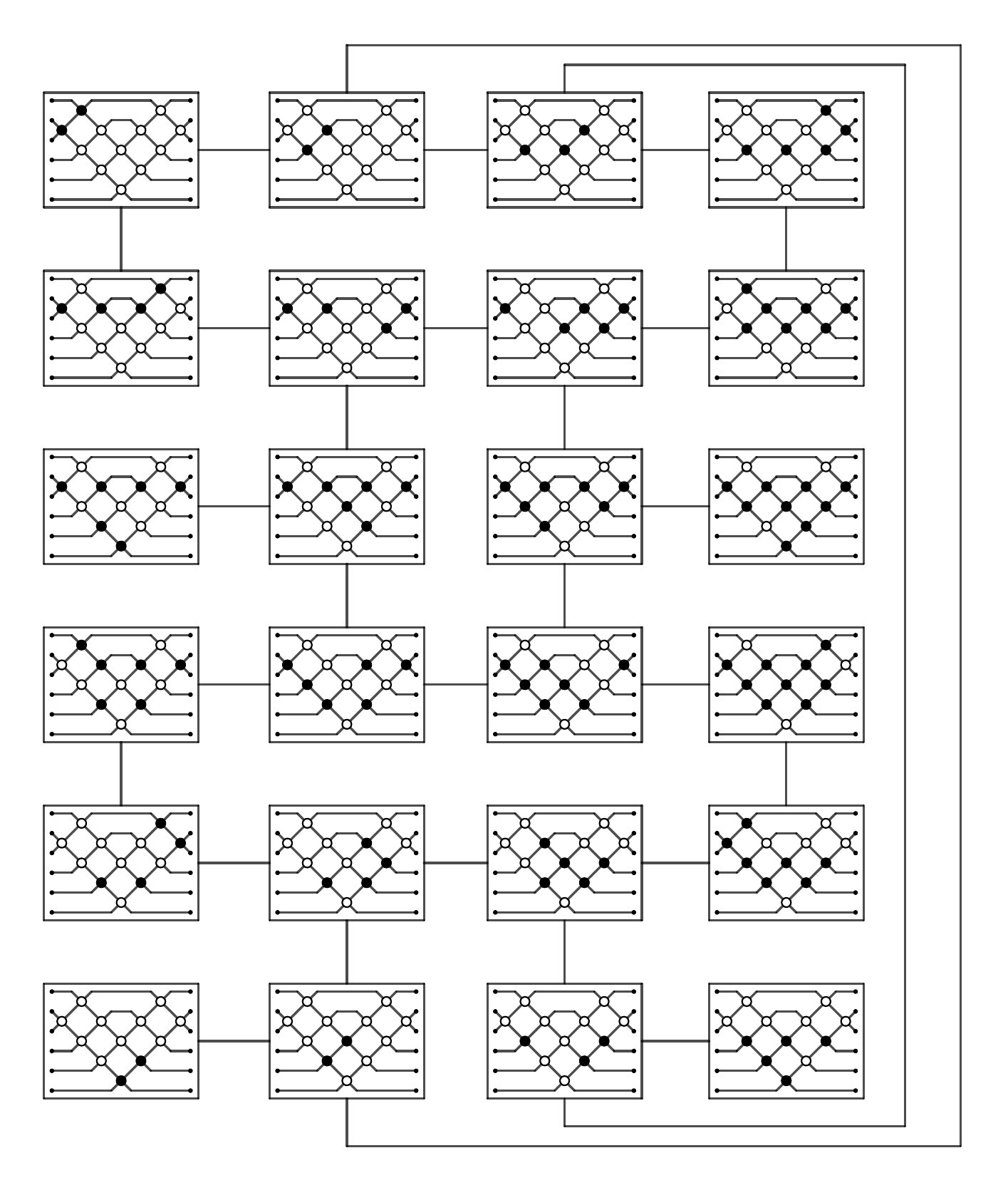}}
    \caption{Connected component homotopically equivalent to $\mathbb{S}^1.$}
    \label{inv12sigma12-2}
\end{figure}

\begin{figure}[H]
   \centerline{    
   \includegraphics[width=0.25\textwidth]{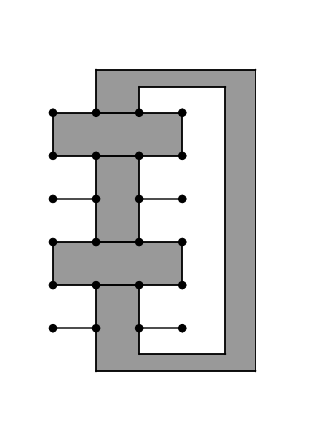}}
    \caption{CW complex homotopically equivalent to $\mathbb{S}^1.$}
    \label{inv12sigma12-2cw}
\end{figure}

\subsection{The New non-Contractible Component}\label{dim3}

Another connected component homotopically equivalent to $\mathbb{S}^1$, and therefore non-contractible, was found with a CW complex of dimension $3$. This component consists of four $3$-cells attached together to form a solid torus. In addition, several $2$-cells are attached as wings, without changing its homotopy type. The component corresponds to the orbits $\mathcal{O}_{\hat{a}_3\acute{\sigma}}$ and $\mathcal{O}_{\hat{a}_1\hat{a}_3\acute{\sigma}}$

Let us go through the step by step construction of this component, 
adding the 3-cells one by one until we attach the last one with the first to generate the solid torus. 
Note that in some cells, we have vertices connected to only one edge. In some of these cases, 
we connect them with an edge in another solid, thus generating the mentioned wings.

    \textbf{Step 1:} First, we have a 3-cell that fills the cube in Figure \ref{inv12sigma12-3.1.2}.

\begin{figure}[H]
   \centerline{    
   \includegraphics[width=0.55\textwidth]{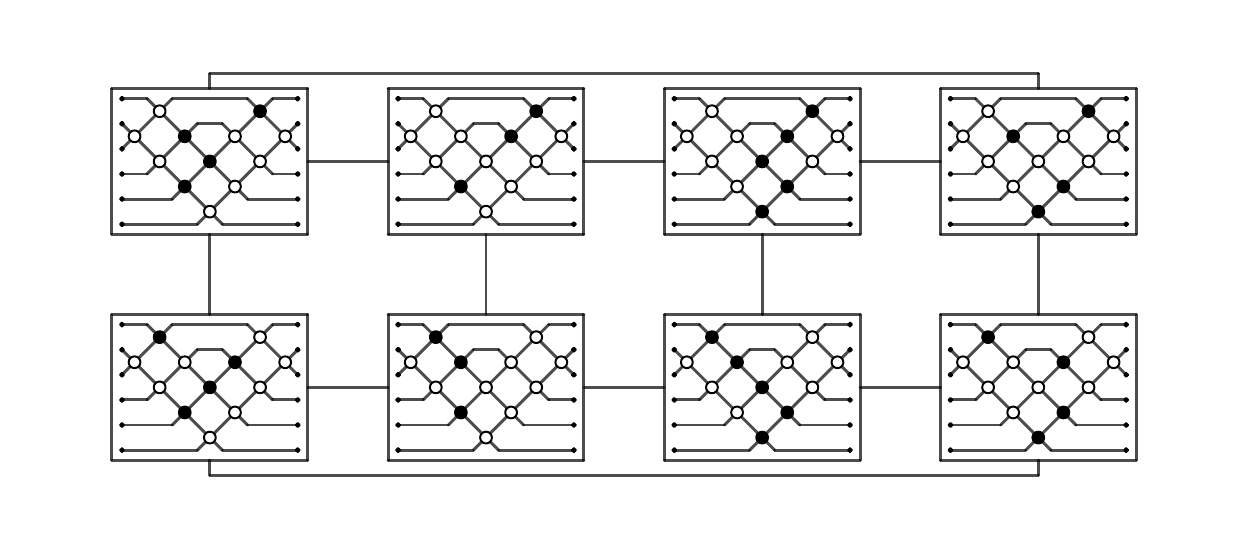}}
    \caption{Cube with ancestry $\varepsilon_1=(\circ\sbd\circ\sbd\sbd\circ\sd\sd\circ\circ\circ\,\circ)$.}
    \label{inv12sigma12-3.1.2}
\end{figure}

Some lower-dimensional cells are attached to the cube, resulting in the structure shown in Figure \ref{inv12sigma12-3.1}.

\begin{figure}[H]
   \centerline{    
   \includegraphics[width=0.55\textwidth]{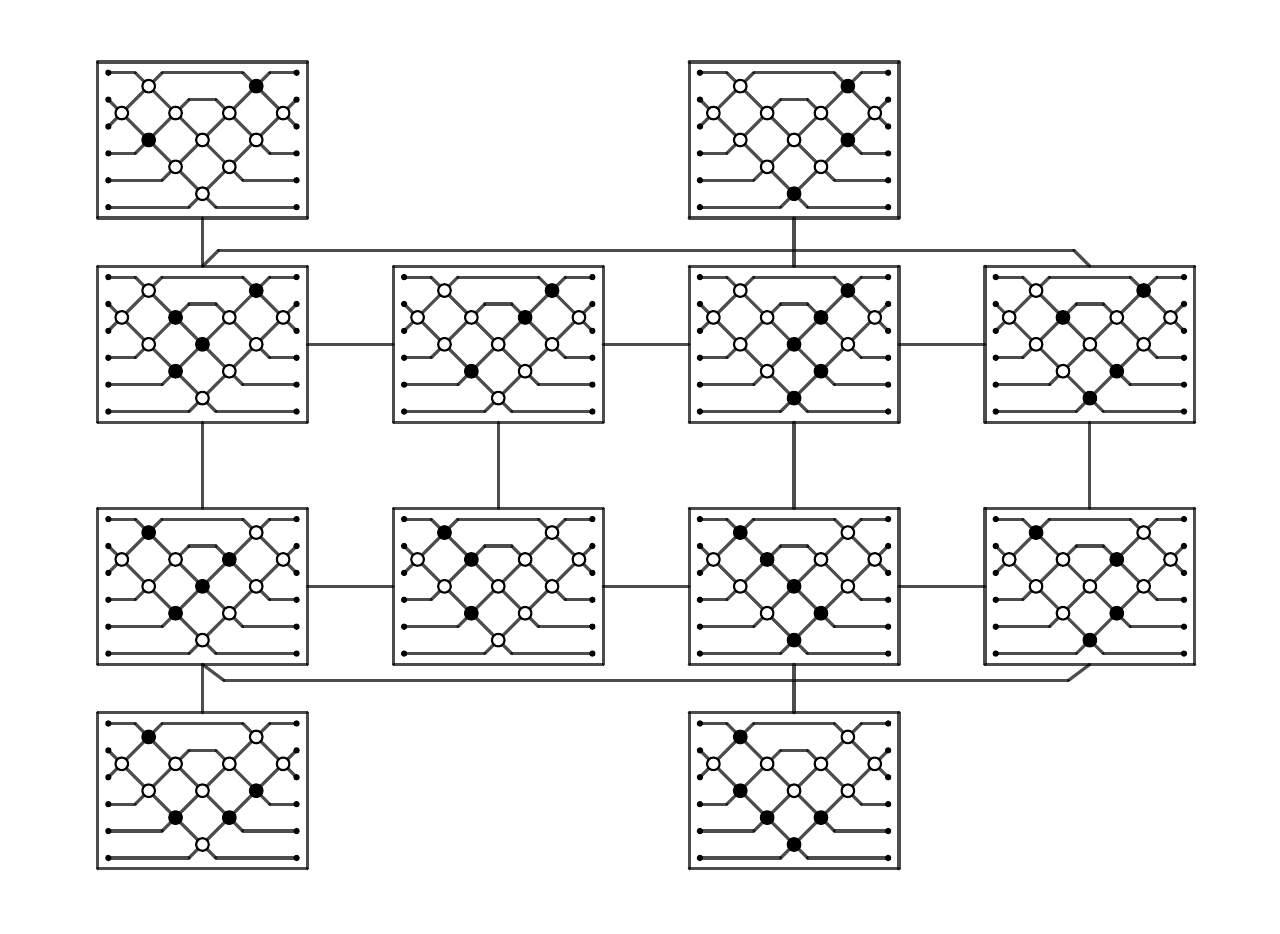}}
    \caption{First part of the CW complex.}
    \label{inv12sigma12-3.1}
\end{figure}

\textbf{Step 2:} Attach a 3-cell that fills the convex solid with 18 faces in Figure \ref{inv12sigma12-3.3.2}. Attachment occurs through the square face with ancestry $\varepsilon_2=~(\circ\sbc\circ\sbd\sbd\circ\sd\circ\circ\sd\circ\,\circ)$.

       \begin{figure}[H]
   \centerline{    
   \includegraphics[width=1.0\textwidth]{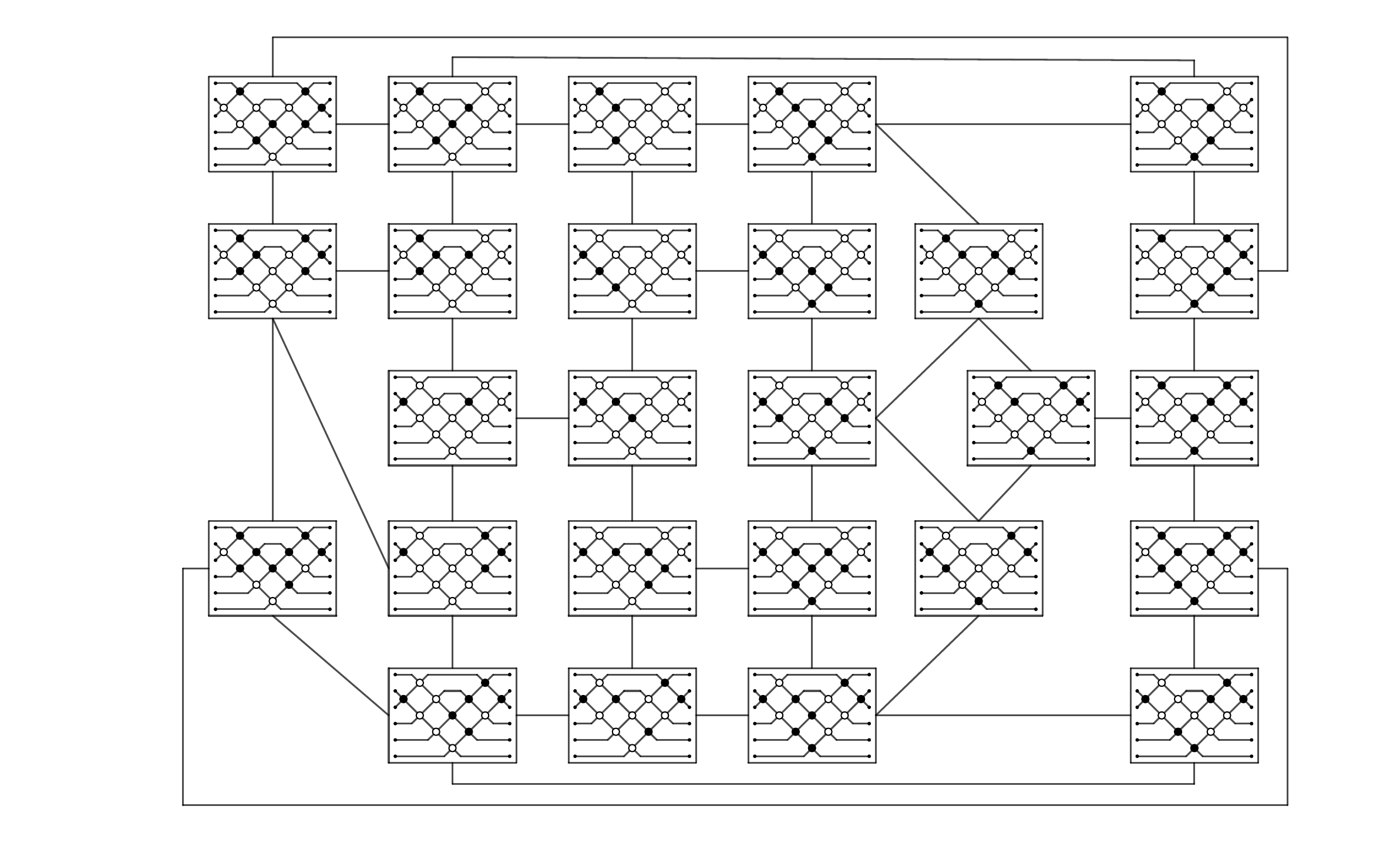}}
    \caption{Convex solid with ancestry $\varepsilon_3=(\sbd\circ\sbd\circ\sbd\circ\circ\circ\circ\sd\sd\,\sd).$}
    \label{inv12sigma12-3.3.2}
\end{figure}

Some lower-dimensional cells are attached to the convex solid, resulting in the structure shown in Figure \ref{inv12sigma12-3.3}.

       \begin{figure}[H]
   \centerline{    
   \includegraphics[width=1.0\textwidth]{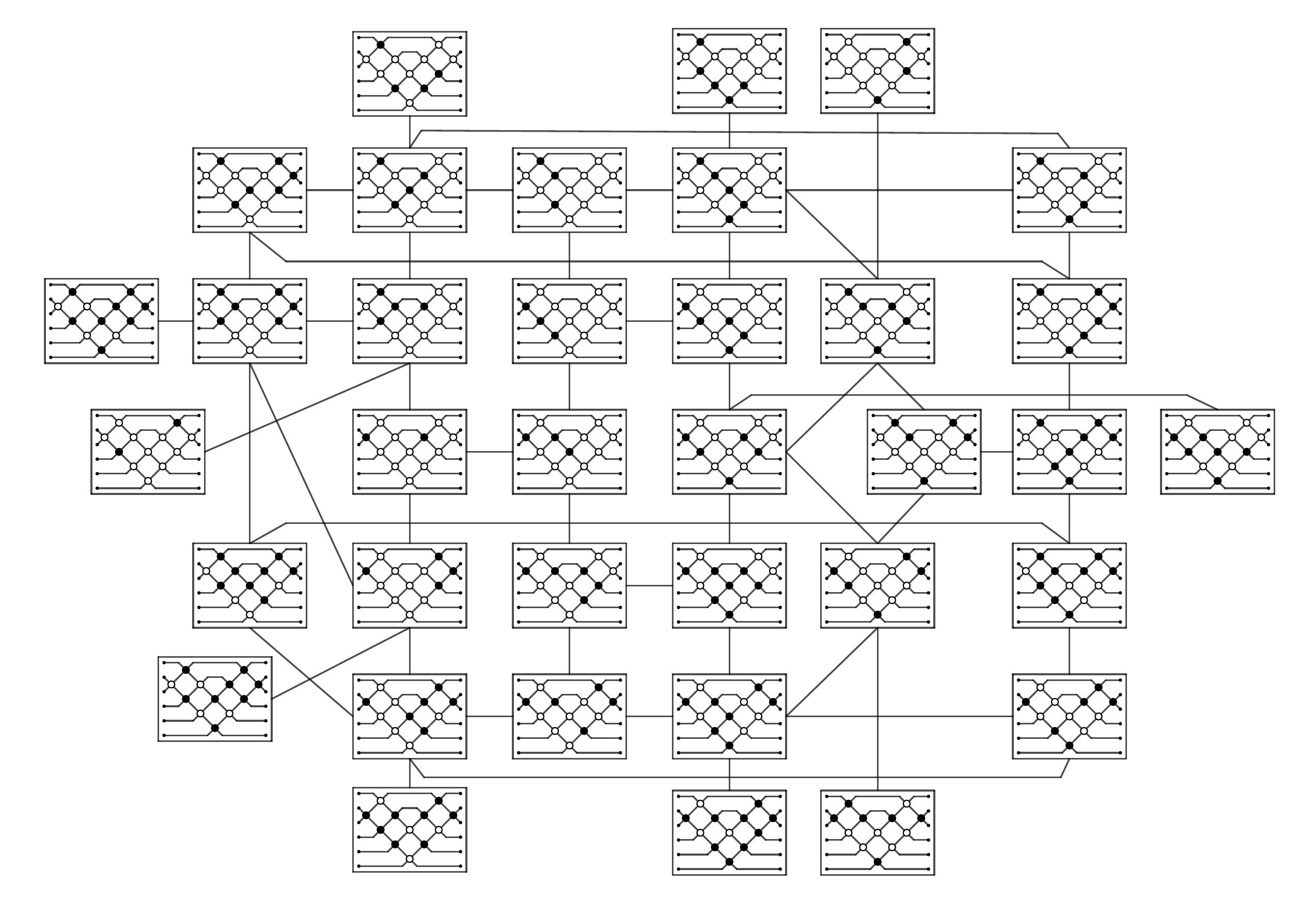}}
    \caption{Second part of the CW complex.}
    \label{inv12sigma12-3.3}
\end{figure}

Following this attachment, two 2-cells appear as wings in the component when we attach the previous two.

       \begin{figure}[H]
   \centerline{    
   \includegraphics[width=0.65\textwidth]{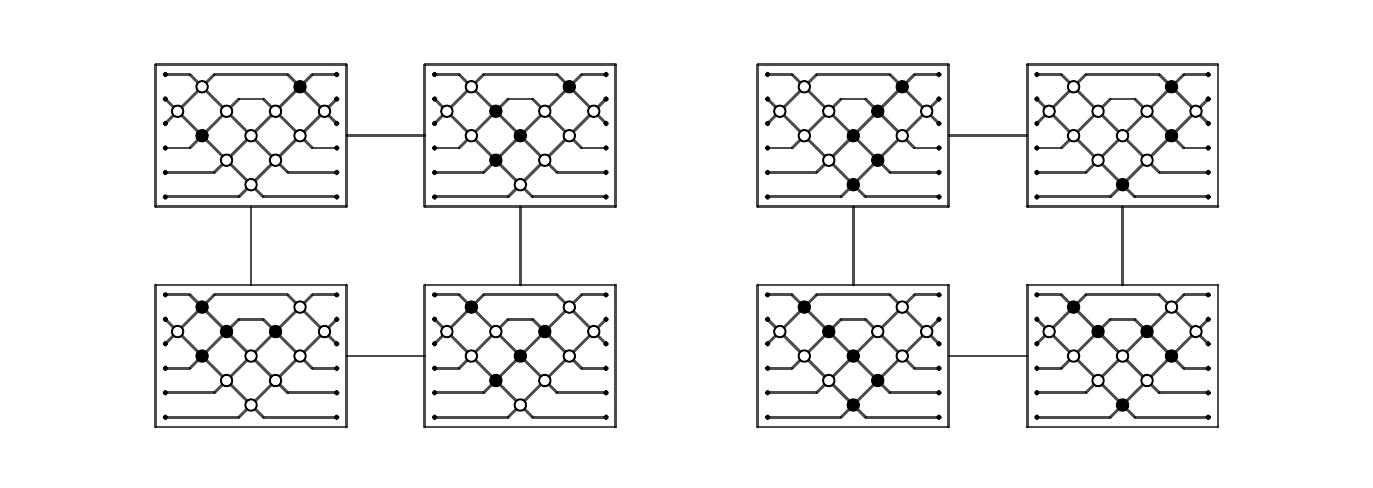}}
    \caption{Cells with dimension 2 with ancestries $\varepsilon_4=(\circ\sbd\sbd\sbc\circ\sd\sbc\sd\circ\circ\circ\,\circ)$ and $\varepsilon_5=~(\circ\sbd\circ\sbc\circ\sbd\circ\sd\sbc\sbc\sd\,\circ)$.}
    \label{inv12sigma12-3.5}
\end{figure}

\textbf{Step 3:} Attach another 3-cell that fills a cube, as shown in Figure \ref{inv12sigma12-3.2}.
Attachment occurs through the square face with ancestry 
$\varepsilon_6=~(\sbc\circ\circ\sbd\sbd\sbc\sd\sbc\sbc\sd\circ\,\sbc)$.

       \begin{figure}[H]
   \centerline{    
   \includegraphics[width=0.55\textwidth]{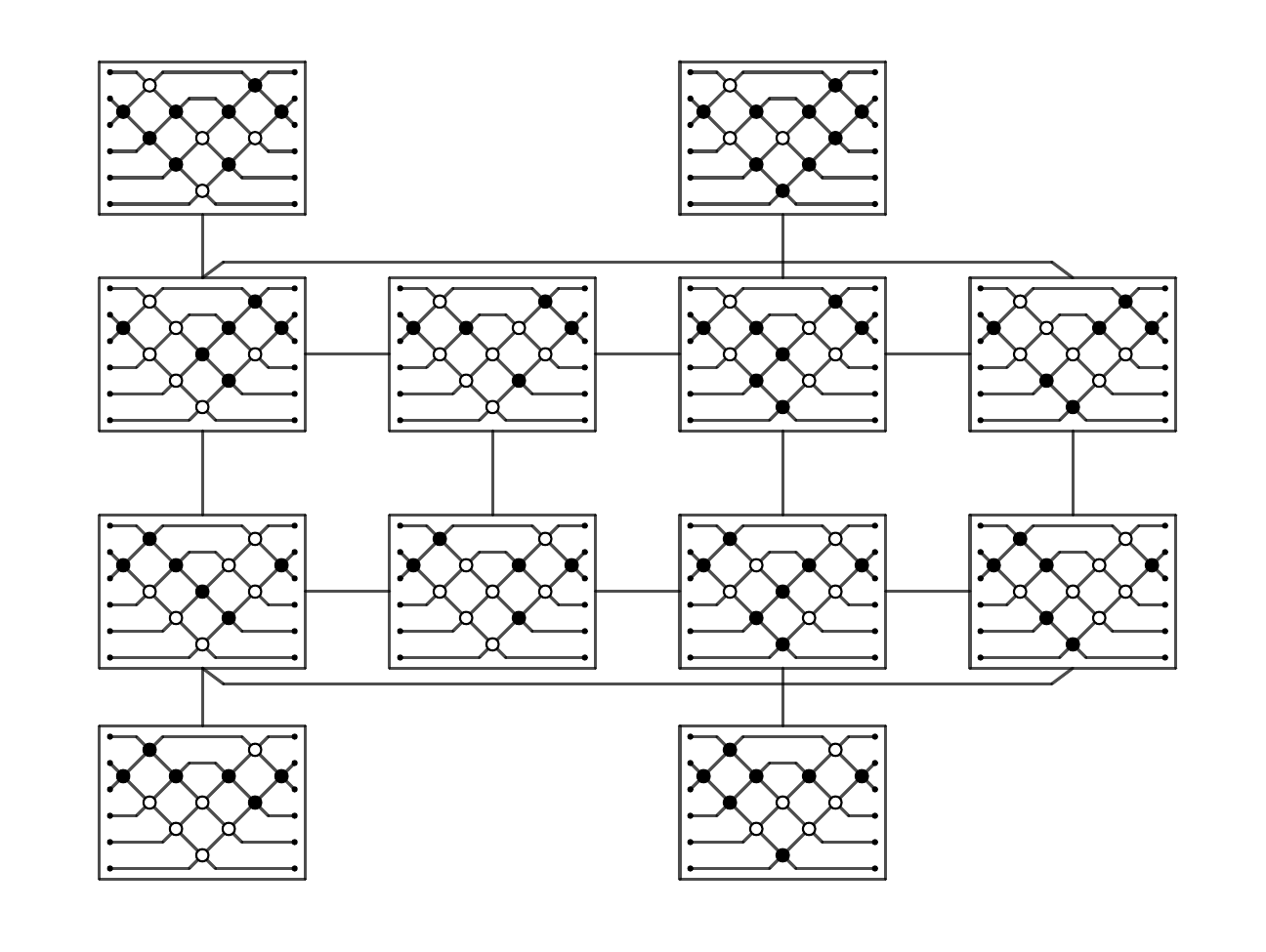}}
    \caption{Third part of the CW complex with ancestry of dimension 3: $\varepsilon_7=~(\sbc\sbd\circ\sbd\sbd\circ\sd\sd\sbc\sd\circ\,\sbc).$}
    \label{inv12sigma12-3.2}
\end{figure}

Following this attachment, similar to the previous case, some 2-cells appear as wings in the component.

       \begin{figure}[H]
   \centerline{    
   \includegraphics[width=0.65\textwidth]{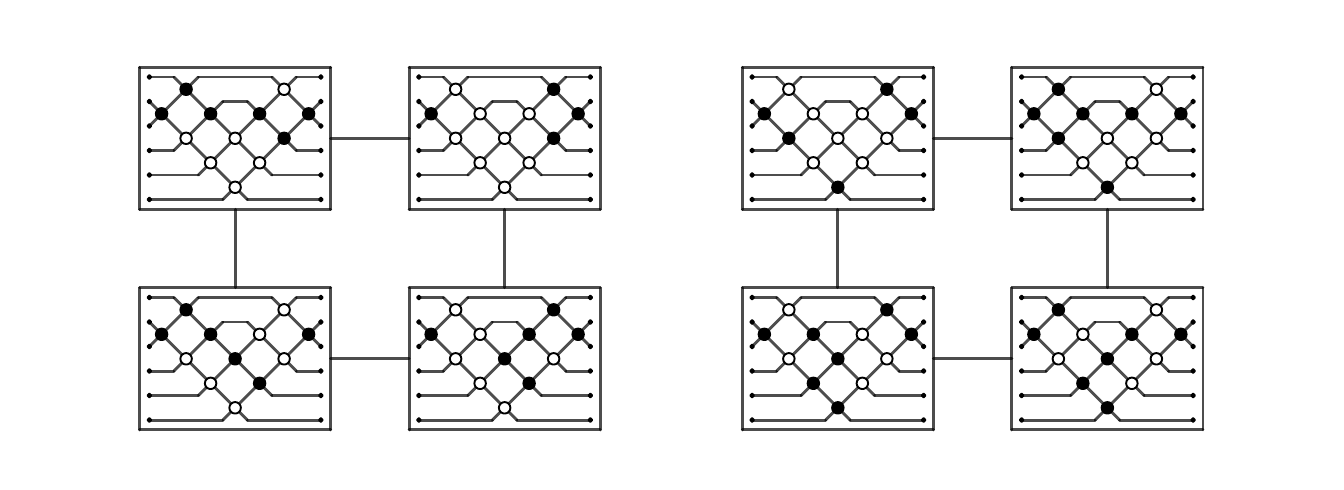}}
    \caption{Cells with dimension 2 with ancestries $\varepsilon_8=(\sbc\sbd\circ\sbc\circ\sbd\circ\sd\circ\sbc\sd\,\sbc)$ and $\varepsilon_9=~(\sbc\sbd\sbd\sbc\circ\sd\sbc\sd\sbc\circ\circ\,\sbc).$}
    \label{inv12sigma12-3.6}
\end{figure}

\textbf{Step 4:} Attach a 3-cell that fills another convex solid with 18 faces, similar to the previous one.
Attachment occurs through the square face with ancestry 
$\varepsilon_{10}=~(\sbc\sbc\circ\sbd\sbd\circ\sd\circ\sbc\sd\circ\,\sbc)$.

       \begin{figure}[H]
   \centerline{    
   \includegraphics[width=1.1\textwidth]{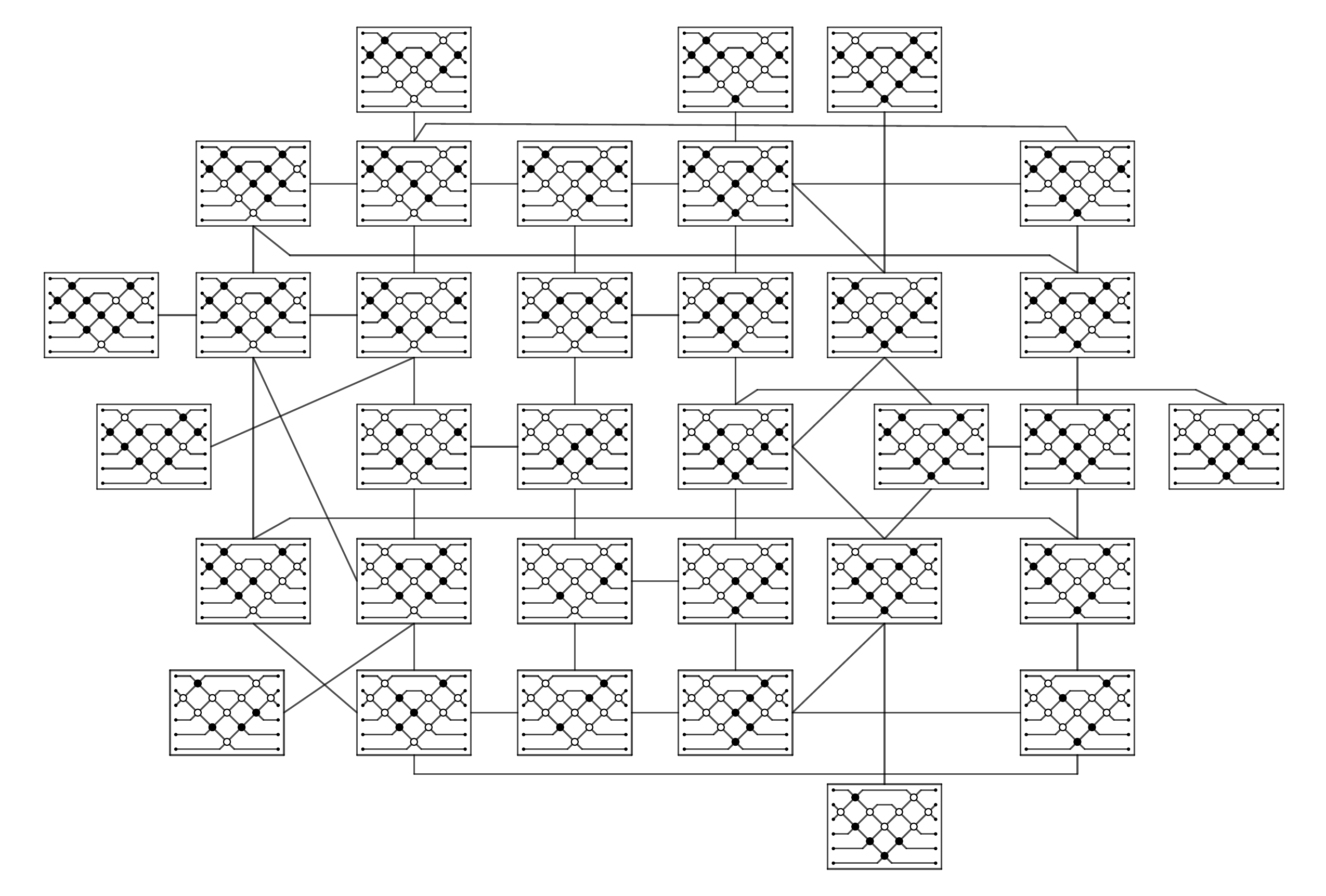}}
    \caption{Fourth part of the CW complex with ancestry of dimension 3: $\varepsilon_{11}=~(\sbd\sbc\sbd\circ\sbd\sbc\circ\sbc\circ\sd\sd\,\sd).$}
    \label{inv12sigma12-3.4}
\end{figure}

Following this attachment, we have some 2-cells that also appear as wings in the component.

       \begin{figure}[H]
   \centerline{    
   \includegraphics[width=0.65\textwidth]{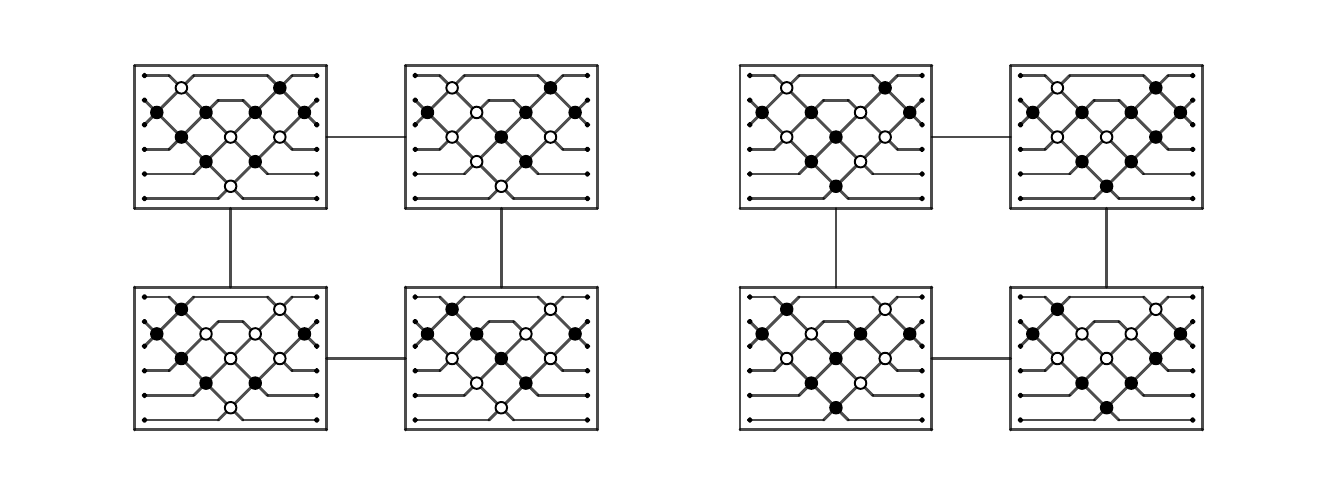}}
    \caption{2-Cells with ancestries $\varepsilon_{12}=(\sbc\sbd\sbd\circ\sbc\sd\circ\sd\circ\sbc\circ\,\sbc)$
    and \\$\varepsilon_{13}=~(\circ\sbc\sbd\circ\circ\sbc\sbd\sbc\sd\sbc\circ\sd\,\sbc)$.}
    \label{inv12sigma12-3.7}
\end{figure}

\textbf{Step 5:} To complete the attachment, the one in Figure \ref{inv12sigma12-3.7} is attached to the cell in Figure \ref{inv12sigma12-3.1}, 
resulting in the formation of the solid torus.
The attachment is realized on the square face with ancestry
$\varepsilon_{14}=(\circ\circ\circ\sbd\sbd\sbc\sd\sbc\circ\sd\circ\,\circ)$.

After this last attachment, we have some 2-cells that appear as wings in the component.

       \begin{figure}[H]
   \centerline{    
   \includegraphics[width=0.65\textwidth]{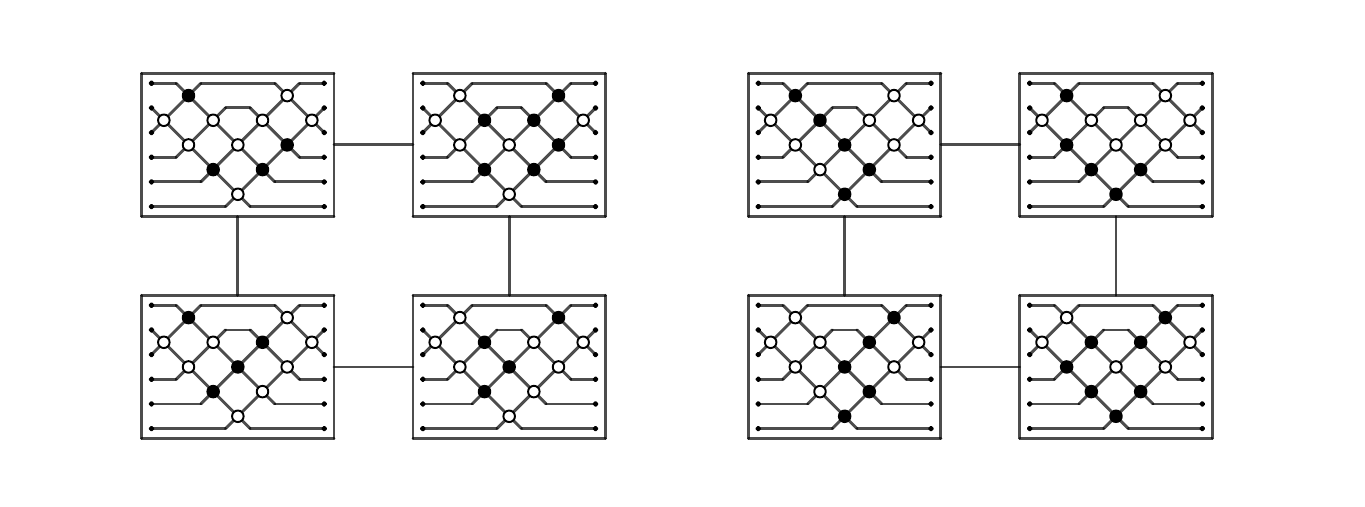}}
    \caption{Cells with dimension 2 with ancestries $\varepsilon_{15}=(\circ\sbd\circ\circ\sbc\sbd\sbc\sd\circ\circ\sd\,\circ)$ and $\varepsilon_{16}=~(\circ\sbd\sbd\circ\sbc\sd\circ\sd\sbc\sbc\circ\,\circ)$.}
    \label{inv12sigma12-3.8}
\end{figure}

Upon completing all the attachments, we have a component that is homotopically equivalent to $\mathbb{S}^1$.
Therefore, $\Bl_{\sigma}$ has 16 components of this type.

For easier visualization, Figure \ref{noncw} first displays the CW complex without the 1-cells and 2-cells attached. It then shows the same CW complex with these cells added, with the red cells representing those not shown in the initial diagram. In this representation, cells of dimension greater than 1 are not filled for clarity.

       \begin{figure}[H]
   \centerline{    
   \includegraphics[width=0.9\textwidth]{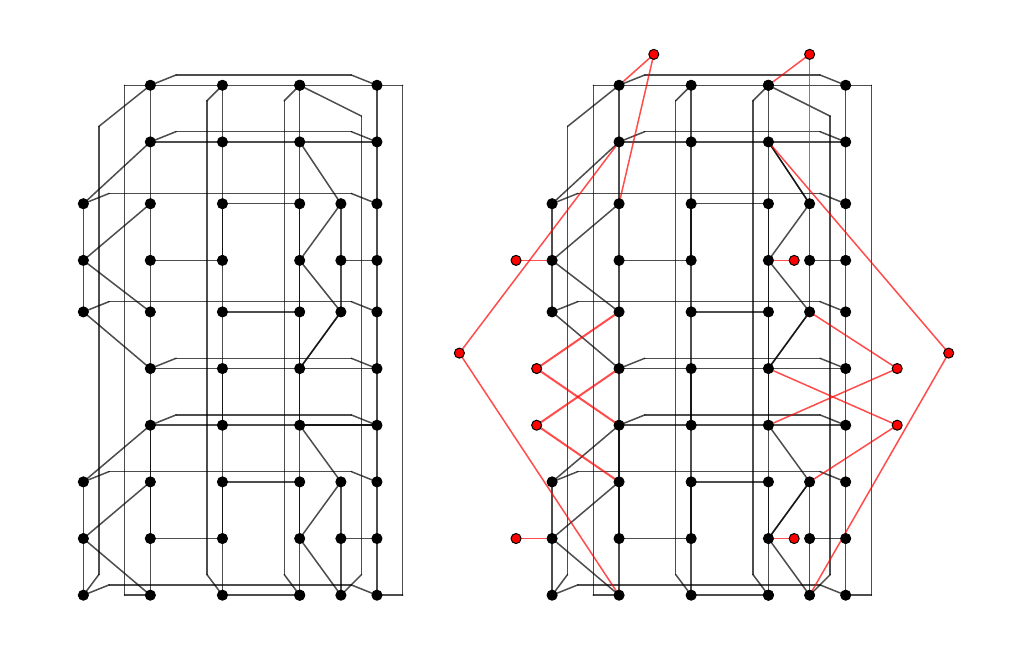}}
    \caption{Non-contractible CW complex.}
    \label{noncw}
\end{figure}

Furthermore, we generate the paths, using Equation \ref{bls0}, that correspond to the edges yielding $\mathbb{S}^1$, by concatenating these paths we obtain the circle. The paths are given by:
\[
\Gamma_i:[-1,1]\to \Lo_6^1,\quad i\in\{1,\ldots,10\},
\]

\[
\footnotesize
\Gamma_1(t)=
\begin{pmatrix}
                1 &  &  &  &  & \\
                0 & 1 &  &  &  &  \\
                0 & 2 & 1 &  &  &  \\
                1 & t+1 & 1 & 1 &  &  \\
                -1 & 1-t & 1 & 0 & 1 &  \\
                0 & 1 & 1 & 1 & 1 & 1 \\
\end{pmatrix}
,\quad
\Gamma_2(t)=
\begin{pmatrix}
                1 &  &  &  &  &  \\
                0 & 1 &  &  &  &  \\
                -1-t & 1-t & 1 &  &  &  \\
                1 & 2 & 1 & 1 &  &  \\
                -1 & t+1 & 2+t & t+1 & 1 &  \\
                0 & 1 & 1 & 1 & 1 & 1 \\
\end{pmatrix},
\]
\[
\footnotesize
\Gamma_3(t)=
\begin{pmatrix}
                1 &  &  &  &  &  \\
                t+1 & 1 &  &  &  &  \\
                -2 & 0 & 1 &  &  &  \\
                -t & 2 & 2+t & 1 &  &  \\
                -1 & 2 & 3 & 2 & 1 &  \\
                0 & 1 & 1 & 1 & 1 & 1 \\
\end{pmatrix}
,\quad
\Gamma_4(t)=
\begin{pmatrix}
                1 &  &  &  &  &  \\
                t-1 & 1 &  &  &  &  \\
                -2 & 0 & 1 &  &  &  \\
                -t-2 & 2 & 2-t & 1 &  &  \\
                -1 & 2 & 1-2t & 2 & 1 &  \\
                0 & 1 & -t & 1 & 1 & 1 \\
\end{pmatrix},
\]
\[
\footnotesize
\Gamma_5(t)=
\begin{pmatrix}
                1 &  &  &  &  &  \\
                0 & 1 &  &  &  &  \\
                -1+t & -1-t & 1 &  &  &  \\
                t-2 & 1-t & 1 & 1 &  &  \\
                -1 & 2 & -1 & 1-t & 1 &  \\
                0 & 1 & -1 & -t & 1 & 1 \\
\end{pmatrix}
,\quad
\Gamma_6(t)=
\begin{pmatrix}
                1 &  &  &  &  &  \\
                0 & 1 &  &  &  &  \\
                0 & -2 & 1 &  &  &  \\
                -1 & -1-t & 1 & 1 &  &  \\
                -1 & 1-t & -1 & 0 & 1 &  \\
                0 & 1 & -1 & -1 & 1 & 1 \\
\end{pmatrix},
\]
\[
\footnotesize
\Gamma_7(t)=
\begin{pmatrix}
                1 &  &  &  &  &  \\
                0 & 1 &  &  &  &  \\
                t+1 & -1+t & 1 &  &  &  \\
                -1 & -2 & 1 & 1 &  &  \\
                -1 & t+1 & -2-t & -1-t & 1 &  \\
                0 & 1 & -1 & -1 & 1 & 1 \\
\end{pmatrix}
,\quad
\Gamma_8(t)=
\begin{pmatrix}
                1 &  &  & &  &  \\
                1+t & 1 &  &  &  &  \\
                2 & 0 & 1 & & & \\
                t & -2 & 2+t & 1 & & \\
                -1 & 2 & -3 & -2 & 1 & \\
                0 & 1 & -1 & -1 & 1 & 1 \\
\end{pmatrix},
\]
\[
\footnotesize
\Gamma_9(t)=
\begin{pmatrix}
                1 & & & & & \\
                1-t & 1 & & & & \\
                2 & 0 & 1 & & & \\
                2+t & -2 & 2-t & 1 & & \\
                -1 & 2 & -1+2t & -2 & 1 & \\
                0 & 1 & t & -1 & 1 & 1 \\
\end{pmatrix}
,\quad
\Gamma_{10}(t)=
\begin{pmatrix}
                1 & & & & & \\
                0 & 1 & & & & \\
                1-t & t+1 & 1 & & & \\
                2-t & -1+t & 1 & 1 & & \\
                -1 & 2 & 1 & -1+t & 1 & \\
                0 & 1 & 1 & t & 1 & 1 \\
\end{pmatrix}.
\]

Note that applying $t=-1$ in $\Gamma_1$ and $t=1$ in $\Gamma_{10}$ results in the same matrix, indicating that the concatenation of these paths forms a closed curve homotopically equivalent to $\mathbb{S}^1$.

For the remaining six components, five are contractible and one has Euler characteristic equal to 1.
See \cite{leal2025homotopy} for the construction of these components.

After completing the analysis of all the components of $\Bl_{\sigma}$ for each $\sigma\in\Sn_6$ with $\inv(\sigma)\leq12$, we arrive at the main result Theorem \ref{theo1}.

\section{Final remarks and open problems}\label{131415}

For permutations $\sigma\in\Sn_6$ with $\inv(\sigma)\geq13$, the difficulty increases significantly. Although we were unable to determine the homotopy type of the components, we do have information about the orbits and the Euler characteristics of their components. 

The maximum dimension of the preancestries for $\sigma$ with $\inv(\sigma)=13$ or $14$ is 5, and for $\inv(\sigma)=15$ it reaches 6. This significantly complicates the visualization of the components and, more importantly, makes interpreting these constructions increasingly challenging and uncertain.

\subsection{Permutations in $\Sn_6$ with $\inv(\sigma)=13$}

There are 14 permutations $\sigma\in\Sn_{6}$ with $\inv(\sigma)=13$. Let us present some important information regarding these permutations.

\begin{enumerate}
    \item There are 6 permutations with 2 orbits, each containing 32 elements. For these permutations, the Euler characteristic of the components is 1, indicating that they are possibly contractible.

    \item There are 8 permutations with five orbits each: three of them containing 16 elements and two with 8 elements.
    
    \begin{enumerate}
        \item Four permutations have four orbits whose components have Euler characteristics equal to 1, compatible with a trivial homotopy type, while one orbit has a component with an Euler characteristic equal to 0, which means a nontrivial homotopy type. Although we have constructed the CW complexes for these components, which suggest they are homotopically equivalent to $\mathbb{S}^1$, the complexity of these constructions makes them difficult to present in full detail in the present paper.

        \item Four permutations exhibit three orbits with components that have Euler characteristics equal to 1, indicating potential contractibility. Two orbits, however, show components with Euler characteristics equal to 2. For these, we have constructed the CW complexes and found that one orbit contains two copies of a contractible CW complex consisting of 56 0-cells, 96 1-cells, 46 2-cells, and 5 3-cells. In the other orbit, there are two distinct components. One of them has 16 0-cells, 16 1-cells, and 1 2-cell, which is contractible. The other one has 128 0-cells, 240 1-cells, 175 2-cells, 52 3-cells, and 6 4-cells. Therefore, this component has Euler characteristic equal to 1, compatible with a trivial homotopy type.
    \end{enumerate}
\end{enumerate}

\subsection{Permutations in $\Sn_6$ with $\inv(\sigma)=14$}

There are 5 permutations $\sigma\in\Sn_{6}$ with $\inv(\sigma)=14$. 

\begin{enumerate}
    \item Four of these permutations have three orbits: two with 16 elements and one with 32 elements. The Euler characteristic of the components for these permutations is 1, consistent with contractibility.

    \item There is one permutation with nine orbits: six of them have components with Euler characteristic equal to 1, indicating a potentially trivial homotopy type. Two orbits have components with Euler characteristic equal to 0, and based on the distribution of the ancestries, we conjecture that these components are homotopically equivalent to $\mathbb{S}^1$. Specifically, these components consist of 256 0-cells, 576 1-cells, 416 2-cells, 100 3-cells, and 4 4-cells. Unfortunately, representing these components graphically exceeds the capabilities of our current tools. The remaining orbit has Euler characteristic equal to 2, for these components, we have constructed the CW complexes and observed that they disconnect, resulting in two copies of a contractible CW complex. The component has 112 0-cells, 216 1-cells 128 2-cells, 24 3-cells and 1 4-cell. The complexity of these constructions makes them difficult to present in full detail in this work.

\end{enumerate}

\subsection{Permutations in $\Sn_6$ with $\inv(\sigma)=15$}

There is only one permutation $\eta\in\Sn_{6}$ with $\inv(\sigma)=15$. Its five orbits are separated into two with 8 elements and three with 16 elements. Furthermore, by Proposition 11.3 in \cite{alves2022onthehomotopy}, $\Bl_{z,\text{thick}},\, z\in\acute{\eta}\Qt_{n+1}$ is nonempty and connected.

From Chapter 15 in \cite{alves2022onthehomotopy}, we already know that there exists a noncontractible component of $\Bl_{\eta}$, with 480 0-cells, 1120 1-cells, 864 2-cells, 228 3-cells and 6 4-cells. Therefore, the component has Euler characteristic equal to 2. The homotopy type of this component remains unknown; additional techniques will be necessary to solve this problem.

The remaining four orbits have components with Euler characteristic equal to 1, consistent with contractibility. One of these orbits, the one with a positive real part, includes one 6-dimensional cell.

\subsection{Permutations in $\Sn_{n+1}$, $n\geq6$}

The following conjectures remain open and will be the focus of future research. We note that Conjecture 1 is strongly motivated by, and follows naturally from, the results presented in this work.

\begin{Conje}
For every $n \geq 6$, there are at least $n - 4$ permutations $\sigma \in \Sn_{n+1}$ with $\inv(\sigma) = 12$ such that $\Bl_{\sigma}$ contains connected components homotopically equivalent to $\mathbb{S}^1$. Furthermore, for each $\sigma \in \Sn_{n+1}$ with $\inv(\sigma) = 12 + k$, where $k \geq 1$, there exist $4^{k}$ connected components that are homotopically equivalent to $\mathbb{S}^{1}$.
\end{Conje}

\begin{Conje}
        Let $\sigma\in\Sn_{n+1}$. If $\inv(\sigma)\leq11$, then every connected component of every set $\Bl_{\sigma}$ is contractible.
\end{Conje}

\printbibliography

\end{document}